\documentclass[]{amsart}
\usepackage{amsmath,amssymb,hyperref,graphicx,mathtools,xcolor}
\def \X{\mathcal{X}} 
\def \gpD{\langle\Delta\rangle} 
\def \gpDe{\langle\Delta^e\rangle} 
\def \CAL{\mathcal C_{\rm AL}} 
\def \dAL{d_{\rm AL}} 
\definecolor{dgreen}{rgb}{0.0, 0.5, 0.0}
\def \co{\colon\thinspace}

\def \Z{\mathbb{Z}}
\def \N{\mathbb{N}}

\newtheorem{theorem}{Theorem}[section]

\newtheorem{lemma}[theorem]{Lemma}

\newtheorem{proposition}[theorem]{Proposition}
\newtheorem{prop}[theorem]{Proposition}

\newtheorem{corollary}[theorem]{Corollary}
\newtheorem{conjecture}[theorem]{Conjecture}

\theoremstyle{definition}
\newtheorem{definition}[theorem]{Definition}

\newtheorem{remark}[theorem]{Remark}

\newtheorem{example}[theorem]{Example}

\newtheorem{notation}[theorem]{Notation}
\newtheorem{question}[theorem]{Question}

\title[Morse elements in Garside groups]{Morse elements in Garside groups are\\ strongly contracting}
\author{Matthieu Calvez} 
\address{Department of Mathematics, Heriot-Watt University, Edinburgh, EH14 4AS, UK}
\address{Instituto de matem\'aticas, Universidad de Valpara\'iso, Gran Breta\~na 1091, 3er piso, Playa Ancha, Valpara\'iso, Chile}
\email{calvez.matthieu@gmail.com}

\author{Bert Wiest}
\address{Bert Wiest, Univ Rennes, CNRS, IRMAR - UMR 6625, F-35000 Rennes, France}
\email{bertold.wiest@univ-rennes1.fr}

\begin{document}

\maketitle

\begin{abstract}
We prove that in the Cayley graph of any braid group modulo its center $B_n/Z(B_n)$, equipped with Garside's generating set, the axes of all pseudo-Anosov braids are strongly contracting. 
More generally, we consider a Garside group~$G$ of finite type with cyclic center.  
We prove that in the Cayley graph of~$G/Z(G)$, equipped with the Garside generators, the axis of any Morse element is strongly contracting. 
As a consequence, we prove that Morse elements act loxodromically on the additional length graph of $G$.
\end{abstract}


\section{Introduction}

For a finitely generated group~$G$, equipped with some fixed finite generating set, and an element $g\in G$ of infinite order, one can study the axis ${\rm axis}(g)=\langle g\rangle$, seen as a set of vertices in the Cayley graph~$\Gamma(G)$. 
There are many different ways of formalizing the idea that this axis might ``look like a geodesic in a hyperbolic space''. 

A particularly weak notion is that $\langle g\rangle$ is quasi-isometrically embedded in~$G$. 
A particularly strong condition is that the axis is \emph{strongly contracting}, which is equivalent to being \emph{strongly constricting}~\cite{ArzhCashenTao}. 
There are many intermediate notions -- for instance the axis could be hyperbolically embedded~\cite{DahmaniGuirardelOsin, OsinAcylHyp, OsinActingAcyl}, 
it could be rank one~\cite{HamenstadtRank1,SistoRandomWalk},
it could be Morse~\cite{DahmaniGuirardelOsin,SistoQuasiConv}, 
it could be contracting in the sense of~\cite{ABD,ArzhCashenTao}, it could have various other contraction and divergence properties~\cite{ArzhCashenGruberHume}, or constriction properties~\cite{ArzhCashenTao}. 

In this paper we will be interested in two of these properties, namely the Morse property and the strong contraction property. Precise definitions can be found in Sections~\ref{S:Morse} and~\ref{S:StrongContraction}.

An element $g$ (or its axis in the Cayley graph $\Gamma(G)$) is said to be \emph{Morse} if this axis is quasi-isometrically embedded in~$\Gamma(G)$,
and if for each pair of constants $(K,L)$, there exists a constant $M_g^{(K,L)}$ such that
every $(K,L)$-quasi-geodesic between two points of the axis travels in an $M_g^{(K,L)}$-neighborhood of the axis.
A remarkable example of the Morse property is the result of Behrstock~\cite{Behrstock} that pseudo-Anosov elements in mapping class groups are Morse -- see also~\cite{SistoQuasiConv,DahmaniGuirardelOsin}.
(In fact, \cite[Theorem~4.2]{DuchinRafi} implies that their axes satisfy the stronger condition of being contracting.)

Saying that the axis of an element $g$ (or indeed any other subset~$A$ of the Cayley graph~$\Gamma(G)$) has the {\emph strong contraction} property means, roughly speaking, that 
there is a constant~$C$ such that taking any ball in~$\Gamma(G)$ disjoint from~$A$, and projecting this ball to $A$ via a closest-point projection, yields a subset of~$A$ of diameter at most~$C$.

One first crucial observation about these two properties is that the Morse property is invariant under quasi-isometry (e.g.\ when looking at the axis of an element in the Cayley graphs of~$G$ with respect to two different generating sets), whereas the strong contraction/constriction property is not. One reason for the failure of quasi-isometry invariance is that these strong properties make reference to actual distances and geodesics, \emph{not} to quasi-geodesics.

There is one well-known family of groups with a very natural family of geodesics between any pair of points in the Cayley graphs, namely the Garside groups. The notions of Garside theory needed in this paper will be recalled in Section~\ref{S:Garside}. We will be interested specifically in \emph{$\Delta$-pure Garside groups of finite type}. 

A Garside group of finite type~$G$ is generated by a finite lattice~$\mathcal D$ with a top element called~$\Delta$. Garside groups are bi-automatic -- in particular, every element $g\in G$ is represented by a unique word in a certain normal form, with letters in $\mathcal D^{\pm 1}$; these normal form words represent \emph{geodesics} in the Cayley graph of~$G$ with respect to~$\mathcal D$. 
We will also require that our Garside groups are $\Delta$-pure, or equivalently Zappa-Sz\'ep indecomposable -- this condition means in particular that the center of~$G$ is infinite cyclic and is generated by some power~$\Delta^e$.

The most famous examples of Garside groups of finite type are the braid groups, and more generally the Artin-Tits groups of spherical type~\cite{CharneyArtinBiautom}. (In this setting, the $\Delta$-pureness condition is equivalent to the defining Coxeter graph being connected.)

Since the infinite subgroup $\gpDe $ of~$G$ is central, there cannot be any elements whose axes in the Cayley graph $\Gamma(G,\mathcal D)$ are Morse or strongly contracting.
Instead, we will study the axes of elements in the Cayley graph of $G$ modulo its center. We will denote $\overline \Gamma=\Gamma(G/Z(G),\,\mathcal D)$ and we will say that an element of a Garside group~$G$ is Morse if its axis in~$\overline\Gamma$ is Morse.

Actually, for our proof it will be technically convenient not to work with the model space~$\overline \Gamma=\Gamma(G,\mathcal D)/\gpDe$, but with the following quasi-isometric variation. We define $\X$ to be the quotient of the Cayley graph $\Gamma(G,\mathcal D)$ under the right $\gpD$-action: $\X=\Gamma(G,\mathcal D)/\gpD$.
We note that this graph $\X$ is the 1-skeleton of the simplicial complex previously considered in~\cite{Bestvina,CMW}. The graph~$\X$ was also studied in~\cite[VIII.3.2]{DDGKM}, under the name~$\mathcal G^0$, and in~\cite{CalvezWiestCurve,CalvezWiestAH}.

Our main result is as follows (see Theorem~\ref{T:StrongContr} for a precise version):

\begin{theorem}\label{T:main}
Suppose $G$ is a $\Delta$-pure Garside group of finite type. Suppose $g$ is an element of~$G$ whose axis in $\X$ (or, equivalently, in~$\overline \Gamma$) is Morse. Then this axis is strongly contracting, both as a subset of $\X$ and as a subset of $\overline \Gamma$.
\end{theorem}

\begin{corollary}\label{C:pAStronglyContr}
Consider the braid group~$B_n$, equipped with the generating set $\mathcal D_{\rm classical}$ or $\mathcal D_{\rm dual}$ coming from its classical or dual Garside structure.
Then in the Cayley graphs $\Gamma(B_n/Z(B_n), \mathcal D_{\rm classical})$ and $\Gamma(B_n/ Z(B_n),\mathcal D_{\rm dual})$, the axis of any pseudo-Anosov braid is strongly contracting.  
\end{corollary}

\begin{remark} 
Here is some context for these results:

(i) \ Rafi and Verberne~\cite{RafiVerberne} have constructed a pseudo-Anosov element of the mapping class group of the five-punctured sphere (which contains $B_4/\langle\Delta^2\rangle$ as a subgroup of index 5), and a generating set for this mapping class group, such that the axis of this pseudo-Anosov element in the corresponding Cayley graph is not strongly contracting. Our Corollary~\ref{C:pAStronglyContr} supports the idea that examples such as that of~\cite{RafiVerberne} can only exist under a pathological choice of generating set.
    
(ii) \ Let $S$ be a surface of finite type, and $\mathcal T(S)$ its Teichm\"uller space, equipped with the Teichmüller metric and the $Mod(S)$-action. 
Minsky proved~\cite{MinskyQuasiProj} that the axis of every pseudo-Anosov element of $Mod(S)$ has the strong contraction property in~$\mathcal T(S)$.
    
(iii) \ There is a hierarchy of contraction properties: 
strongly contracting implies contracting which in turn implies sublinearly contracting, which is equivalent to being Morse~\cite{ArzhCashenGruberHume}; neither of the implications is an equivalence~\cite{RafiVerberne,BradyTran}.
However, under a CAT(0)-hypothesis there \emph{is} a strong converse~\cite{Sultan,CashenMorseStrongContr}: Morse axes in CAT(0)-spaces are strongly contracting. 
Thus Theorem~\ref{T:main} is an indicator that Garside groups have CAT(0)-like behaviour, and thus gives further evidence that the answer to the following question may be affirmative.
\end{remark}

\begin{question}
Are all Garside groups CAT(0)? (Note that it is not even known whether all braid groups are CAT(0) \cite{BradyMccammond,HaettelKielakSchwer, Jeong}.)
\end{question}

Our Theorem~\ref{T:main} begs the question which elements of well-known Garside groups \emph{are} Morse. For braid groups, we know the answer from~\cite{Behrstock}: it's the pseudo-Anosov elements. For other irreducible Artin-Tits groups of spherical type, however, the question is open. The authors hand-constructed in~\cite{CalvezWiestAH} some elements in each such group whose axes are strongly contracting, and in particular Morse. We believe, however, that all plausible candidates for being Morse really are Morse:

\begin{conjecture}\label{C:AllCandidatesMorse}
An element~$a$ of an irreducible Artin-Tits groups of spherical type~$A$ is Morse if and only if its image in $A/Z(A)$ has virtually cyclic centralizer.
\end{conjecture}

An application of our results (indeed, the authors' original motivation for this research) concerns the \emph{additional length graph} $\CAL(G)$~\cite{CalvezWiestCurve,CalvezWiestAH,CalvezWiestSurvey}. To any Garside group~$G$ one can associate a $\delta$-hyperbolic graph~$\CAL(G)$ whose relation to~$G$ is  loosely analogous to the curve graph's relation to the mapping class group. Indeed, if $B_n$ is the $n$-strand braid group, then $\CAL(B_n)$ is conjectured to be quasi-isometric to the curve graph of the $(n+1)$-punctured sphere.

Using Theorem~\ref{T:main} we can prove (see Theorem~\ref{T:MorseLoxOnCAL} and Corollary~\ref{C:MorseDiamInf})

\begin{proposition}\label{P:MorseLoxOnCAL}
Suppose 
$G$ is a $\Delta$-pure Garside group of finite type. 
If $g$ is a Morse element of~$G$, then the action of~$g$ on~$\CAL(G)$ is loxodromic and weakly properly discontinuous. In particular, if $G$ contains a Morse element then $\CAL(G)$ has infinite diameter.
\end{proposition}

The plan of this paper is as follows.
In Section~\ref{S:Garside} we review some, mostly standard, elements from Garside theory which we will need. 
In Section~\ref{S:Morse} we recall the definition of the Morse property and prove a first contraction property for the axes of Morse elements in Garside groups.
In Section~\ref{S:Projection} we define, in a Garside-theoretical fashion, a projection to the axis of any element satisfying a Garside-theoretical rigidity condition. For elements satisfying both the rigidity and the Morse condition, we strengthen our previous contraction result, and deduce that our projection is uniformly close to any closest-point projection.
Section~\ref{S:StrongContraction} contains a precise definition of the strong contraction property and the proof of Theorem~\ref{T:main}.
In Section~\ref{S:LoxOnCAL} we present the applications of the main result to the additional length graph.


\section{Garside groups}\label{S:Garside}

The notion of a Garside group stems from Garside's approach to solving the conjugacy problem in the braid groups~\cite{Garside}. Soon generalized to Artin-Tits groups of spherical type~\cite{BrieskornSaito,Deligne}, this approach was first axiomatized in~\cite{DehornoyParis,Dehornoy} and thoroughly studied over the first decade of the 2000s. The book~\cite{DDGKM} provides a  comprehensive account of what is now called ``Garside theory". 

\begin{definition}
Let $G$ be a group; $G$ is a \emph{Garside group} with \emph{Garside structure} $(G^{+},\Delta)$ if $G^{+}$ is a submonoid of $G$ such that $G^{+}\cap {G^{+}}^{-1}=\{1\}$ and there exists an element $\Delta\in G^{+}$ with the following properties:
\begin{itemize}
    \item[(1)] The partial order relations $\preccurlyeq$ and $\succcurlyeq$ on $G$ defined by 
    \begin{itemize}
        \item $x\preccurlyeq y$ ($x$ is a \emph{prefix} of $y$) if and only if $x^{-1}y\in G^{+}$,
        \item $x\succcurlyeq y$ ($y$ is a \emph{suffix} of $x$) if and only if $xy^{-1}\in G^{+}$  
    \end{itemize}
    are lattice orders on $G$; that is, all $x,y\in G$ admit a unique greatest common prefix $x\wedge y$, a unique greatest common suffix $x\wedge ^{\!\Lsh} y$, a unique least common right multiple $x\vee y$ and a unique least common left multiple $x\vee^{\!\Lsh} y$.
    \item[(2)] The set $\mathcal D=\{x\in G^{+}, x\preccurlyeq \Delta\}=\{x\in G^{+}, \Delta\succcurlyeq x\}$ generates $G^{+}$ as a monoid and $G$ as a group.
    \item[(3)] For all $x\in G^{+}\setminus\{1\}$,
    $$\|x\| = \sup\{k\,|\, \exists\, a_1,\ldots,a_k\in G^{+}\setminus\{1\}\  \text{such that}\  x = a_1 \cdots a_k\}<\infty.$$
\end{itemize}
\end{definition}

The elements of $G^{+}$ are called \emph{positive}, $\Delta$ is the \emph{Garside element} and the elements of $\mathcal D$ are called \emph{simple}. The elements $x$ of $G^{+}$ such that $\|x\|=1$ are called \emph{atoms} and they form a subset of $\mathcal D$.

Given a simple element $s$, its \emph{right-complement} $\partial(s)$ is defined by $\partial(s)=s^{-1}\Delta$ and its \emph{left-complement} is defined by $\partial^{-1}(s)=\Delta s^{-1}$; both $\partial(s)$ and $\partial^{-1}(s)$ belong to $\mathcal D$. 
Conjugation by $\Delta$ will be denoted by $\tau$ (that is, for $g\in G$, $\tau(g)=\Delta^{-1}g\Delta$); notice that for every simple element $s$, 
$\partial^2(s)=\partial(\partial(s))=\tau(s)$.

We shall make the additional assumption that $G$ is \emph{of finite type}, i.e.\ that the set of simple elements~$\mathcal D$ is finite. 
In this case, it follows that $\tau$ has finite order, and we will denote this order by~$e$. 
The element $\Delta^e$ is then central in~$G$.
We shall further assume that $G$ is \emph{$\Delta$-pure}. 
This property was defined in~\cite{Picantin}, and shown in~\cite[Theorem 39]{GebhardtTawn} to be equivalent to indecomposability as a Zappa-Sz\'ep product.
All the reader needs to know about $\Delta$-pure Garside groups of finite type is that their center
is cyclic, generated by~$\Delta^e$, i.e.\ $Z(G)=\gpDe$ (\cite{Picantin}). 
For instance, Artin-Tits groups of spherical type are Garside groups, and they are $\Delta$-pure if and only if the defining Coxeter graph is connected~\cite[Proposition 4.7]{Picantin}.

\begin{notation}
Throughout this paper, $G$ denotes a $\Delta$-pure Garside group of finite type. We denote $e$ the positive integer such that $Z(G)=\gpDe $. 
When we talk about Cayley graphs, it is always understood that the generating set is obtained from the set of simple elements~$\mathcal D$. 
We will use the notation $\Gamma$ for the Cayley graph $\Gamma(G,\mathcal D)$.
Also, we denote $\overline\Gamma = \Gamma(G/Z(G))$, the Cayley graph of $G$ modulo its center, with respect to the generators which are the images of $\mathcal D$ in~$G/Z(G)$. The corresponding graph metrics will be denoted $d_{\Gamma}$ and $d_{\overline\Gamma}$ respectively. 
\end{notation}

To each element of $G$ we associate three integer numbers as follows. 
\begin{definition}\cite[Section 1]{ElRifaiMorton}
Let $g\in G$. The \emph{infimum} of $g$ is defined by $\inf(g)=\max\{r\in \Z, \ \Delta^r\preccurlyeq g\}$, the \emph{supremum} of $g$ is defined by $\sup(g)=\min\{s\in \Z,\ g\preccurlyeq \Delta^s\}$ and the \emph{canonical length} of $g$ is defined by $\ell(g)=\sup(g)-\inf(g)$.
\end{definition}

It is well-known that each element of $G$ can be written uniquely as an irreducible fraction involving elements of $G^{+}$-- the letters $D$ and $N$ in the following stand for ``denominator" and ``numerator" respectively. 
\begin{lemma}
\cite[Lemma 4.4]{CharneyInjectivity}
\label{L:CharneySplittings}
Let $g\in G$. 
\begin{itemize}
    \item[(i)](Left-fraction). There is a unique pair of positive elements $(D_l(g), N_l(g))$ such that $D_l(g)\wedge N_l(g)=1$ and $g=D_l(g)^{-1}N_l(g)$. In particular, if 
    $c$ is any positive element such that $cg$ is positive, then $c\succcurlyeq D_l(g)$.
    \item[(ii)](Right-fraction). There is a unique pair of positive elements $(D_r(g), N_r(g))$ such that
    $D_r(g)\wedge^{\!\Lsh} N_r(g)=1$ and $g= N_r(g)D_r(g)^{-1}$.
    In particular, if $c$ is any positive element such that $gc$ is positive, then $D_r(g)\preccurlyeq c$. 
    \item[(iii)] We have the equalities $\inf(D_l(g))=\inf(D_r(g))$, $\sup(D_l(g))=\sup(D_r(g))$, 
    $\inf(N_l(g))=\inf(N_r(g))$ and $\sup(N_l(g))=\sup(N_r(g))$.
\end{itemize}
\end{lemma}

Moreover, every element of $G$ can be decomposed as follows. Recall that a pair of simple elements $(s,t)\in \mathcal D^2$ is \emph{left-weighted} if $\partial(s)\wedge t=1$ and \emph{right-weighted} if $\partial^{-1}(t)\wedge^{\!\Lsh} s=1$. 

\begin{proposition}[\cite{Adyan},\cite{Dehornoy}, Section~3]\label{P:LNF}
Let $g\in G$. Let $p=\inf(g)$, $r=\ell(g)$.
\begin{itemize}
\item[(i)] 
There exists a unique decomposition $g=\Delta^p s_1\cdots s_r$, where
$s_1,\ldots, s_r\in \mathcal D\setminus\{1,\Delta\}$ and for every $1\leqslant i\leqslant r-1$, $(s_i,s_{i+1})$ is left-weighted. 
\item[(ii)] Similarly, there exists a 
unique decomposition $g=s'_r\cdots s'_1\Delta^p$, where $s'_1,\ldots, s'_r\in \mathcal D\setminus\{1,\Delta\}$ 
and for every $1\leqslant i\leqslant r-1$, $(s'_{i+1},s'_{i})$ is right-weighted. 
\end{itemize}
These decompositions are called \emph{left normal form} and \emph{right normal form} of $g$, respectively. 
\end{proposition}

Considering the latter normal forms together with the fractional decompositions of Lemma~\ref{L:CharneySplittings} yields a slightly different notion of normal form. 

\begin{definition}\cite[Proposition 3.9]{Dehornoy}
Let $g\in G$. If ${\inf(g)<0<\sup(g)}$ and $D_l(g)=a_1\cdots a_r$ and $N_l(g)=b_1\cdots b_s$ are left normal forms, 
the \emph{left mixed normal form} of $g$ is the writing 
$g=a_r^{-1}\cdots a_1^{-1}b_1\cdots b_s$; 
similarly, the \emph{right mixed normal form} is the writing $g=b'_s\cdots b'_1a_1'^{-1}\cdots a_r'^{-1}$, where $a'_r\cdots a'_1$ and $b'_s\cdots b'_1$ are the respective right normal forms of $D_r(g)$ and $N_r(g)$. 
If $\inf(g)\geqslant 0$, then the left (right, respectively) mixed normal form of~$g$ coincides with 
the left (right, respectively) normal form of~$g$ given by Proposition~\ref{P:LNF}. 
If $\sup(g)\leqslant 0$, then the left (right, respectively) mixed normal form of~$g$ is the formal inverse of the left (right, respectively) normal form of $g^{-1}$ given by Proposition~\ref{P:LNF}.
\end{definition}

These mixed normal forms have an important geometric meaning:

\begin{lemma}\cite[Lemma 3.1]{CharneyMeier}\label{L:CharneyGeod}
In $\Gamma$, the Cayley graph of $G$ with respect to $\mathcal D$, 
mixed normal forms are geodesics. 
\end{lemma}

Finally, we shall need one more Garside-theoretical definition: 

\begin{definition}\label{D:Rigid}
Let $x\in G$ with right normal form $x= x_r\cdots x_1\Delta^p$.
We say that~$x$ is \emph{right-rigid} if its preferred simple suffix $\mathfrak p^{\!\Lsh}(x):=\tau^p(x_1)\wedge^{\!\Lsh}\partial^{-1}(x_r)$ is trivial. In particular, if $\inf(x)=0$, then for $k\geqslant 1$, the right normal form of $x^k$ consists of the concatenation of~$k$ copies of the right normal form of $x$. 
\end{definition}

Our aim is, of course, to study the geometry of $G$ modulo its center $Z(G)=\gpDe $. 
However, it is technically far more convenient not to study the Cayley graph $\overline \Gamma = \Gamma(G/Z(G))$ directly, but a very closely related graph, which we will denote~$\X$. 
This graph~$\X$ is the 1-skeleton of ``Bestvina's normal form complex" considered in~\cite{Bestvina,CMW}, and it has been described previously in~\cite{CalvezWiestCurve} and \cite[Chapter VIII, Section 3.2]{DDGKM}. We recall the definition:

\begin{notation}
We denote~$\X$ the quotient of $\Gamma=\Gamma(G)$ by the right-action of~$\gpD$: 
$$ \X = \Gamma(G,\mathcal D)/\gpD$$
\begin{itemize}
    \item The vertices of $\X$ are  the left-cosets of $G$ modulo $\gpD$ : $\{g\langle\Delta\rangle,\  g\in G\}$. 
    Each vertex $g\langle \Delta\rangle$ of $\X$ possesses a unique \emph{distinguished representative} $\underline g$, which is by definition the representative satisfying~$\inf(\underline g)=0$:
    given $g\in G$, we have $\underline g = g\Delta^{-\inf(g)}$.
    We denote by $\ast$ the vertex $\gpD$, whose distinguished representative is the trivial element of $G$.
    \item Two vertices $g\gpD$ and $h\gpD$ of $\X$ are connected by an edge if there is a proper simple element $s\in \mathcal D$ such that $\underline g s\in h\gpD$; 
    this is equivalent to the existence of a proper simple element $t$  such that $\underline h t\in g\gpD$. 
\end{itemize}
\end{notation}

The following provides more precise information about adjacent vertices of $\mathcal X$:

\begin{lemma}\cite[Lemma 3.4]{Bestvina} 
\label{L:AdjacencyInX}
Suppose that $g\gpD$ and $h\gpD$ are adjacent vertices of $\X$. 
Then there exists a proper simple element $s\in \mathcal D\setminus\{1,\Delta\}$ such that one of the following holds:
\begin{itemize}
    \item $\underline g s= \underline h$ (and in this case we have $\underline h\,\partial s = \underline g\Delta$), or
    \item $\underline h s= \underline g$ (and in this case we have $\underline g\,\partial s = \underline h \Delta$)
\end{itemize}
\end{lemma}

\begin{notation}\label{Not.}
We denote by $d_\X$ the graph metric on the graph $\X$; for $g,h\in G$, we sometimes denote $d_{\mathcal X}(g,h)$ for $d_{\mathcal X}(g\gpD,h\gpD)$.
Note that the groups $G$ and $G/Z(G)$ act isometrically by left-translations on $\X$ by $g\cdot (g'\gpD)=(gg')\gpD$. 
\end{notation}

The spaces $\X$ and~$\overline \Gamma$ are very closely related:
\begin{proposition}\label{P:XandGammaBarAreQI}
There is an \emph{isometric} embedding $\iota\co \X \hookrightarrow \overline \Gamma$ with
$\lfloor\frac{e}{2}\rfloor$-dense image (i.e.\ every vertex of~$\overline \Gamma$ is at distance at most $\lfloor\frac{e}{2}\rfloor$ from a vertex belonging to~$\iota(\X)$). 
In particular, $\iota$ is a quasi-isometry.
\end{proposition}

\begin{proof}
If $g,h\in G$ represent adjacent vertices of~$\X$, then we know from Lemma~\ref{L:AdjacencyInX} that $\underline g$ and $\underline h$ represent adjacent vertices in the Cayley graph $\Gamma$ of~$G$, and thus also in the Cayley graph $\overline \Gamma$ of $G/\gpDe$.
This means that the map 
$$
\iota\co \{\text{vertices of }\X\} \hookrightarrow \{\text{vertices of }\overline\Gamma\}, \ g\gpD\mapsto \underline g\gpDe 
$$
sends adjacent vertices to adjacent vertices, and thus induces a well-defined and 1-Lipschitz map of graphs $\iota\co \X\hookrightarrow \overline\Gamma$.

In the other direction, there is a natural projection map $$p\co \{\text{vertices of }\overline \Gamma\} \to \{\text{vertices of }\X\}, \ 
g\gpDe \mapsto g\gpD$$ 
which induces a well-defined map of graphs $p\co \overline\Gamma \to \X$.
Both $\iota$ and~$p$ are 1-Lipschitz, and $p\circ\iota={\rm id}_\X$. This implies that $\iota$ is an isometric embedding. 

Now we look at the opposite composition 
$$\iota\circ p\co \overline\Gamma \to \overline\Gamma, \ g\gpDe  \mapsto \underline g\gpDe $$ 
We observe that $d_{\overline\Gamma}(g\gpDe ,\underline g\gpDe ) < \lfloor\frac{e}{2}\rfloor$, which means that $\iota\circ p$ is at distance~$\lfloor\frac{e}{2}\rfloor$ from ${\rm id}_{\overline\Gamma}$; thus the image of~$\iota$ is $\lfloor \frac{e}{2}\rfloor$-dense.
\end{proof}

We now recall from~\cite{CalvezWiestCurve} the notion of a preferred path between two vertices in~$\X$. 

\begin{definition}\label{D:PreferredPaths}
Given any vertex $g\gpD$ of $\X$, let $s_1\cdots s_r$ be the left normal form of $\underline g$; 
the \emph{preferred path} between $\ast$ and $g\gpD$ is the path  
$$\ast,\, s_1\gpD,\, \ldots\, ,\, (s_1\cdots s_n)\gpD=g\gpD.$$ 
The preferred path between the vertices $g\gpD$ and $h\gpD$ of $\X$ is 
the $\underline g$-left translate of the preferred path between $\ast$ and $\underline g^{-1}h \gpD$.
It is denoted by $A(g\gpD, h\gpD)$, or simply $A(g,h)$, keeping in mind that the definition does not depend on the particular representatives $g$ and $h$.
\end{definition}
 
We record some basic properties of the preferred paths (recall that for $g,h\in G$, $g\wedge h$ is the unique greatest common prefix of $g$ and $h$):
\begin{proposition}\label{P:PreferredPath}
\begin{itemize}
\item[(i)] Let $g,h\in G$; let $p=\underline g \wedge \underline h$.
 The preferred path $A(g,h)$ is the concatenation of the preferred paths $A(g,p)$ and $A(p,h)$. 
 \item[(ii)] Preferred paths are symmetric: for all $g,h\in G$, $A(g,h)$ is the reverse of $A(h,g)$. 
 \item[(iii)] Preferred paths are preserved by left-translation: for all $g,h,k\in G$, 
 $A(kg,kh)=kA(g,h)$. 
 \item[(iv)] Preferred paths are geodesics in $\mathcal X$ and for all $g,h\in G$, $d_{\mathcal X}(g,h)=d_{\Gamma}(\underline g,\underline h)$.
 \item[(v)] Balls in~$\X$ are convex: 
 if $g,h\in G$ and $k\gpD\in A(g,h)$, then $d_\X(\ast,k\gpD)\leqslant \max(d_\X(\ast,g\gpD),d_\X(\ast,h\gpD))$.
\end{itemize}
\end{proposition}

\begin{proof}
(i) and (ii) correspond to \cite[Lemma 4]{CalvezWiestCurve} and \cite[Lemma 5]{CalvezWiestCurve}, respectively. 
For (iii), write the left normal form of $\underline{\underline g^{-1}h}$ as $z_1\cdots z_r$, so that $A(g,h)$ is by definition the path 
$$g\gpD=\underline g\gpD,\, \underline g z_1\gpD,\, \ldots\, ,\, \underline g z_1\cdots z_r\gpD = h\gpD.$$
Note that $\underline{kg}= k\underline g\Delta^j$ (with $j=\inf(g)-\inf(kg)$); 
therefore the left normal form of $\underline{(\underline{kg})^{-1}kh}$ is $\tau^j(z_1)\cdots \tau^j(z_r)$.
So by definition, $A(kg,kh)$ is the path 
$$kg\gpD= \underline{kg}\gpD , \,\underline{kg}\tau^j(z_1)\gpD, \,\ldots\, ,\, \underline{kg} \tau^j(z_1\cdots z_r) \gpD,$$
but for all $1\leqslant i \leqslant r$, we have $$\underline{kg}\tau^j(z_1\cdots z_i)= k\underline g \Delta^{j}\Delta^{-j}z_1\cdots z_i\Delta^j= k\underline g z_1\cdots z_i\Delta^{j},$$
thus $A(kg,kh)$ is the path
$$ k g\gpD, \, k \underline g z_1\gpD, \,\ldots \, ,\, k \underline g z_1\cdots z_r\gpD,$$
the $k$-left translate of $A(g,h)$, as claimed.

To see (iv), recall first that mixed normal forms are geodesics in $\Gamma$ (Lemma~\ref{L:CharneyGeod}). Let $g,h\in G$. It is shown in the proof of \cite[Lemma 4]{CalvezWiestCurve} that the path $A(g,h)$ in~$\X$ has the exact length of the mixed normal form of $\underline g ^{-1}\underline h$, say $r$. Suppose that there was a shorter path in $\mathcal X$ between $g\gpD$ and $h\gpD$, that is a sequence of vertices $g\gpD=g_0\gpD, g_1\gpD,\ldots, g_k\gpD=h\gpD$, with $k<r$. 
Then by Lemma~\ref{L:AdjacencyInX} there are simple elements $s_1,\ldots, s_k$ such that
for $1\leqslant i\leqslant k$, we have $\underline{g_i}= \underline{g_{i-1}} s_i$ or $\underline{g_i}= \underline{g_{i-1}}s_i^{-1}$; setting $t_i$ to be $s_i$ or $s_i^{-1}$ accordingly, we obtain that 
$\underline h=\underline {g_k}=\underline{g_0}t_1\cdots t_{k}=\underline g t_1\cdots t_k$, where each $t_i$ is either a simple element or the inverse of a simple element. This contradicts the fact that mixed normal forms are geodesics in~$\Gamma$. 

Now, let us prove (v). In view of (iv), the distances involved are the respective lengths of the left normal forms of $\underline k$, $\underline g$ and $\underline h$. 
Let $p= \underline g \wedge \underline h$; write $\underline g = p a $, $\underline h =p b$ (with $a,b\in G^{+}$) and let $a=a_1\cdots a_r$, $b=b_1\cdots b_s$ the respective left normal forms. By the proof of \cite[Lemma 4]{CalvezWiestCurve}, the distinguished representatives of the vertices along the path $A(g,h)$ are (in this order) 
$$\underline g = pa_1\cdots a_r,\;\ldots\;,pa_1,p, pb_1,\;\ldots,\; pb_1\cdots b_s=\underline h.$$
Any of these is a prefix of $\underline g $ or $\underline h$ thus its left normal form is at most as long as that of $\underline g$ or $\underline h$ and the claim is proved.  
\end{proof}

\begin{lemma}[Fellow traveller property]\label{L:AFellowTraveller}
Suppose that $g,g', h, h'\in G$, such that $d_{\X}(g,g')=1$ and $d_{\X}(h,h')=1$. 
Then the set of vertices along the path $A(g,h)$ and the set of vertices along the path $A(g,h')$ in~$\X$ are at Hausdorff distance~1. 
Also, the analogue statement holds for the paths $A(g,h)$ and $A(g',h)$.
\end{lemma}

\begin{proof}
By symmetry of preferred paths (Proposition~\ref{P:PreferredPath}(ii)), the second statement follows from the first. After a left translation, we may assume $g=1$, so we must show the claim for $A(\ast,h\gpD)$ and $A(\ast, h'\gpD)$. 
We may assume from the hypothesis (and Lemma~\ref{L:AdjacencyInX}) that there exists a simple element $s$ such that 
$\underline h s= \underline{h'}$. 
Let $z_1\cdots z_r$ be the left normal form of $\underline h$. 
The vertices along $A(\ast,h\gpD)$ are $z_1\cdots z_i \gpD$, for $0\leqslant i \leqslant r$.

\begin{figure}[htb]
\begin{center}
\def\svgwidth{12cm}
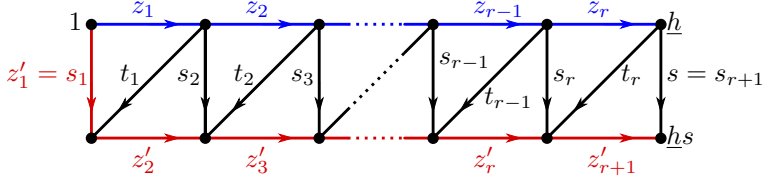
\end{center}
\caption{The normal form paths of $\underline h$ (in blue) and of $\underline h s$ (in red) stay at Hausdorff distance~1 in the Cayley graph of~$G$.}\label{F:FellowTraveller}
\end{figure}

Now, the left normal form of $\underline h'=\underline h s$ may be obtained, for instance, from the algorithm given in \cite[Proposition 1.2]{GebhardtGM}, which is illustrated in Figure~\ref{F:FellowTraveller}. 
It goes as follows. Set $s_{r+1}=s$; for $i=r,\ldots,1$, define recursively the simple elements $t_i=\partial(z_i)\wedge s_{i+1}$, $s_i=z_it_i$ and $z'_{i+1}=t_i^{-1}s_{i+1}$. Finally define $z'_1=s_1$. Then $z'_1\cdots z'_r z'_{r+1}$ (or $z'_1\cdots z'_r$, if $z'_{r+1}=1$) is the left normal form of $\underline h$. 

We then see that for $i=1,\ldots,r$, we have $z'_1\cdots z'_i=z_1\cdots z_it_{i}$, hence showing that the $(i+1)$st vertex along $A(\ast, h'\gpD)$ is at distance 1 from the $(i+1)$st vertex along $A(\ast,h\gpD)$.
\end{proof}

\begin{lemma}[Concatenation of normal form paths]\label{L:ConcatQuasiGeod} 
Let $g,h,k\in G$ be such that 
$\underline g \preccurlyeq \underline h \preccurlyeq \underline k$. 
Then 
\begin{itemize}
    \item[(i)] The concatenation of the paths $A(g,h)$ and $A(h,k)$ is a $(2,0)$-quasi-geodesic connecting $g\gpD$ to $k\gpD$.
    \item[(ii)] Suppose in addition that the paths $A(g,h)$ and $A(g,k)$ have the same length; then the concatenation of $A(h,k)$ and $A(k,g)$ is a $(2,0)$-quasi-geodesic connecting $h\gpD$ to $g\gpD$.
\end{itemize}
\end{lemma}
\begin{proof}
Let $a=\underline g^{-1}\underline h$, $b=\underline h^{-1}\underline k$ and $c=\underline g^{-1}\underline k$. 
Under our hypothesis, $\inf(a)=\inf(b)=\inf(c)=0$ and $c=ab$. Denote by $a=a_1\cdots a_r$ and $b=b_1\cdots b_s$ the respective left normal forms. 

(i) Denote $\alpha$ the concatenation of paths under consideration. Let $p\gpD$ be the $k$th vertex on $A(g,h)$, for some $1\leqslant k\leqslant r$; let $q\gpD$ be the $l+1$st vertex on $A(h,k)$, for some $1\leqslant l\leqslant s$. 
The distance $d_{\mathcal X}(p,q)$ is the length of $A(p,q)$, that is $\sup(\underline{\underline p^{-1}q})=\sup(a_{k}\cdots a_rb_1\cdots b_{l})$. 
As the subwords $a_{k}\cdots a_r$ and $b_1\cdots b_{l}$ are in left normal form, we obtain $d_{\mathcal X}(p,q)\geqslant \max\{r-k+1,l\}$. Note that the portion of $\alpha$ between $p\gpD$ and $q\gpD$ has length $r-k+1+l$. We then obtain
$$\frac{1}{2}(r-k+1+l)\leqslant \max\{r-k+1,l\}\leqslant d_{\mathcal X}(p,q)\leqslant r-k+1+l\leqslant 2(r-k+1+l),$$
hence showing that $\alpha$ is a $(2,0)$-quasi-geodesic.

(ii) The extra hypothesis is saying that $\sup(a)=\sup(c)$. Consider the vertices $g'\gpD=\underline h^{-1} g\gpD$, $h'\gpD=\underline h^{-1}h\gpD=\ast$ and $k'\gpD=\underline h^{-1}k\gpD$; 
our goal is to show that the concatenation of $A(h',k')$ and $A(k',g')$ is a $(2,0)$-quasi-geodesic connecting $h'\gpD$ to $g'\gpD$. 
Let us compute the respective distinguished representatives: we have $\underline h'=1$, $\underline k'=b$ and 
$$\underline g'=\underline{\underline  h^{-1}g}=\underline{a^{-1}}=a^{-1}\Delta^{\sup(a)}.$$ 
Moreover, $a^{-1}\Delta^{\sup(a)}=b(c^{-1}\Delta^{\sup(a)})=b (c^{-1}\Delta^{\sup(c)})$ and $c^{-1}\Delta^{\sup(c)}$
is positive, 
so $b\preccurlyeq \underline g'$. We thus have $\underline h'\preccurlyeq \underline k'\preccurlyeq \underline g'$ and we are in a position to apply Lemma~\ref{L:ConcatQuasiGeod}(i) to obtain the desired claim.
\end{proof}


\section{Morse elements in Garside groups}\label{S:Morse}

In this section we will define Morse elements in Garside groups, and prove some preliminary results on them.
First, we set some notation and we recall the definition. We recall that throughout, $G$~denotes a $\Delta$-pure Garside group of finite type.

\begin{notation}
Let $g\in G$. 
\begin{itemize}
    \item The \emph{axis} of $g$ in $\mathcal X$ is the set of vertices $\{g^t\gpD,\ t\in \mathbb Z\}$. 
    \item The \emph{axis} of $g$ in $\overline\Gamma$ is the set of vertices $\{g^t\gpD,\ t\in \mathbb Z\}$.
    \end{itemize}
\end{notation}

\begin{remark}
We caution the reader that two elements $g$ and $g\Delta$ may have completely different axes in $G / Z(G)$, and thus in $\X$. For instance in braid groups, it may very well happen that some element $g$ is pseudo-Anosov while $g\Delta$ is reducible.
\end{remark}

\begin{definition}\label{D:Morse1}
(a) Let $(X,d)$ be a metric space, and let $\gamma\co \mathbb Z\to X$ be a map. We say that $\gamma$ (or equivalently its image) is \emph{Morse} if 
\begin{enumerate}
    \item[(i)] $\gamma$ is a quasi-isometric embedding, and
    \item[(ii)] for every pair $(\Lambda,K)$ with $\Lambda\geqslant 1$, $K\geqslant 0$, there is a number $M^{(\Lambda,K)}$ (the Morse constant) such that every $(\Lambda,K)$-quasi-geodesic in $X$ connecting two points $\gamma(i), \gamma(j)$ ($i,j\in \mathbb Z$) of the image of~$\gamma$ remains in the $M^{(\Lambda,K)}$-neighbourhood of the image of~$\gamma$ in~$X$.
\end{enumerate}

(b) Let $H$ be a group generated by a finite set~$S$. An element $h\in H$ is \emph{Morse} if the map $\gamma\co \mathbb Z\to Cay(H,S)$, $t\mapsto h^t$ is Morse in the sense of~(a). 
In this situation, given $\Lambda\geq 1$ and $K\geq 0$ we denote by $M_{h}^{(\Lambda,K)}$ the associated Morse constant.
\end{definition}

\begin{notation}
Since many quasi-geodesics in this paper will be $(2,0)$-quasi-geodesics, we will use the simplified notation $M_h=M_h^{(2,0)}$.
\end{notation}

Since a $\Delta$-pure Garside group of finite type has an infinite-cyclic center, it cannot contain any Morse elements. We adapt the definition as follows, keeping in mind Proposition~\ref{P:XandGammaBarAreQI} which says that the projection $\overline \Gamma\to \X$, $g\gpDe  \mapsto g\gpD$ is a quasi-isometry and the fact that the Morse property is invariant under quasi-isometry. 

\begin{definition}\label{D:Morse2}
We say that an element $g$ of a $\Delta$-pure Garside group of finite type~$G$ is \emph{Morse} if any of the following equivalent condition holds:
\begin{itemize}
    \item[(i)] the image of $g$ in $G / Z(G)$ is Morse in the sense of Definition~\ref{D:Morse1}(b), 
    \item[(ii)] the axis of $g$ in $\overline\Gamma$ is Morse in the sense of Definition~\ref{D:Morse1}(a).
    \item[(iii)] the axis of $g$ in $\mathcal X$ is Morse in the sense of Definition~\ref{D:Morse1}(a).
    \end{itemize}
\end{definition}



\begin{example}
Pseudo-Anosov braids are Morse;
indeed, their projections to the group $B_n/Z(B_n)$, which is a finite-index subgroup of the mapping class group of an $(n+1)$-times punctured sphere, are pseudo-Anosov mapping classes in this group.
By~\cite{Behrstock}, these are Morse.
\end{example}

The following is a key technical result -- recall that the notion of right-rigidity is introduced in Definition~\ref{D:Rigid}.

\begin{proposition}\label{P:PowerConjToRigid}
Every Morse element of $G$ has a power which is conjugate to a right-rigid element.
Moreover, this right-rigid element can be required to be of the form $\Delta^{e\cdot m}x$, with $m\in\mathbb Z$, $\inf(x)=0$, and $x$ right-rigid.
\end{proposition}

In order to prepare the proof of Proposition~\ref{P:PowerConjToRigid}, we need the following lemma. It is well-known to experts, but we were not able to find a proof in the literature.

\begin{lemma}\label{L:Centralizer}
Let $H$ be a group generated by a finite set $S$. 
If $h\in H$ is a Morse element in the sense of Definition~\ref{D:Morse1}(b), then $\langle h \rangle $ has finite index in the centralizer $Z(h)$ in $H$.
\end{lemma}
\begin{proof}
Denote by $d$ the word-distance associated to $S$ in $H$. 
Let $A\geqslant 1$ and $B\geqslant 0$ such that 
$t\mapsto h^t$ is an $(A,B)$-quasi-isometric embedding of $\mathbb Z$ in $Cay(H,S)$.
Let $M=M_h^{(4A,B)}$. 

We claim that if $zh=hz$, then $d(z,\langle h\rangle )\leqslant M$. Granted this claim, we then see that for all $z\in Z(h)$, there is some $m\in \mathbb Z$ so that $d(z,h^m)=d(1,z^{-1}h^m)\leqslant M$, that is the coset $z^{-1}\langle h\rangle$ has a representative of word length at most~$M$. It follows that each coset of $Z(h)$ modulo~$\langle h \rangle$ has a representative of word length at most $M$, so there are only finitely many cosets. 

To prove the claim, assume on the contrary that $d(z,\langle h\rangle)=K>M$, for some $z\in Z(h)$. Let $r$ such that $d(z, h^r)=K$; up to replacing $z$ by $h^{-r}z$, we may assume that $K=d(z,\langle h\rangle )=d(z,1)$. 
Choose $N$ big enough so that $d=d(h^{-N},h^N)>4K$.
Then we piecewise define a map $\gamma\co \{0,\ldots,2K+2N\} \to Cay(H,S)$ as follows. 
\begin{itemize}
    \item For $0\leqslant t \leqslant K$, $\gamma(t)=\gamma_1(t)$ follows a geodesic from $h^{-N}$ to $zh^{-N}$.
    \item For $K\leqslant t\leqslant K+2N$, $\gamma(t)=\gamma_2(t)=zh^{t-K-N}$.
    \item For $K+2N\leqslant t \leqslant 2K+2N$, $\gamma(t)=\gamma_3(t)$ follows a geodesic from 
    $zh^{N}$ to~$h^N$. 
\end{itemize}

To conclude, we shall observe that $\gamma$ defines a $(4A,B)$-quasi-geodesic with endpoints $h^{-N}$ and $h^N$ on the axis of $h$, hence contradicting the hypothesis that $h$ is Morse (as the vertex $z=\gamma(K+N)$ does not lie in the $M$-neighborhood of $\langle h^n, n\in \mathbb Z\rangle$).

Let $0\leqslant s,t\leqslant 2K+2N$. The upper bound 
$d(\gamma(s),\gamma(t))\leqslant A|s-t|+B$ follows immediately from the fact that $\gamma$ is a juxtaposition of geodesics and an $(A,B)$-quasi-geodesic. To obtain the lower bound, we consider two cases. 

First, let $0\leqslant s\leqslant K$ and $K\leqslant t\leqslant K+2N$ (the 
situation is symmetric for $K+2N\leqslant s\leqslant 2K+2N$).
Because $zh^{-N}=\gamma(K)$ realizes the minimum possible distance between $h^{-N}$ and $\{zh^{n}, n\in \Z\}$, we have 
$d(\gamma(s),\gamma(t))\geqslant d(\gamma(s),\gamma(K))$. 
Then we have 
$$d(\gamma(s),\gamma(K))+d(\gamma(K),\gamma(t))\leqslant 2d(\gamma(s),\gamma(K)) + d(\gamma(s),\gamma(t))\leqslant 3d(\gamma(s),\gamma(t))$$
and the lower bound $d(\gamma(s),\gamma(t))\geqslant \frac{1}{3A}|s-t|-\frac{B}{3}$ follows.

Second, let $0\leqslant s\leqslant K$ and $K+2N\leqslant t\leqslant 2K+2N$. By our choice of $N$, we have
$$ d(\gamma(s),\gamma(t))\geqslant 
\frac{d}{2}\geqslant
\frac{d}{4}+K\geqslant 
\frac{d}{4}+\frac{K}{2}\geqslant \frac{d}{4}+\frac{1}{4}d(\gamma(s),zh^{-N})+
\frac{1}{4}d(\gamma(t),zh^N)$$
and we obtain the lower bound $d(\gamma(s),\gamma(t))\geqslant \frac{1}{4A}|t-s|-\frac{B}{4}$.
\end{proof}

\begin{proof}[Proof of Proposition~\ref{P:PowerConjToRigid}] 
The proof of the first sentence is identical to the proof of \cite[Theorem 3.23, Corollary 3.24]{Birman-Gebhardt-GM}, replacing the words ``pseudo-Anosov braid'' by 
``Morse element of~$G$". Indeed, being Morse is a property stable by conjugacy and powers, a Morse element is never a root of a central element, 
and every element of $G$ commuting with a Morse element $g$ has a common power with $g$, up to multiplication by a central power of $\Delta$ (Lemma~\ref{L:Centralizer}). 
These are the only properties of pseudo-Anosov braids that are used in the proofs of the quoted results. 

For the proof of the second sentence, we simply remark that for a \emph{rigid} element~$g$ and an integer~$k$,
$$ \inf(g^k)=k\cdot \inf(g) $$
Thus by taking a further power, we can achieve that the infimum of the rigid conjugate is a multiple of~$e$.
\end{proof}

\begin{remark} There is an alternative proof of Proposition~\ref{P:PowerConjToRigid},
which does not use Lemma~\ref{L:Centralizer} or the paper~\cite{Birman-Gebhardt-GM}, but which is rather reminiscent of the ``pumping lemma'' \cite[Theorem 1.2.13]{CEHLPT}.
The idea is that the right normal form of~$g^n$ has to stay close to the axis of the Morse element~$g$. Now right normal forms belong to a language recognized by a finite state automaton; thus for large enough~$n$ this right normal form has a middle segment with periodic behaviour. This periodic segment is a right-rigid conjugate of a power of~$g$. We leave the details as an exercise.
\end{remark}

\begin{notation}
From here on, all diagrams in this paper will take place in~$\X$ (not in the Cayley graph~$\Gamma(G)$ or in $\overline \Gamma=\Gamma(G/Z(G))$). Also, in the diagrams, we simplify the notation, labelling a vertex~$g$ if it is represented by a group element~$g$ -- strictly speaking, it should be labelled~$g\gpD$.
\end{notation}

\begin{prop}[Preferred paths stay close to the axis]\label{P:Matthieu'sLemma}
Suppose $x$ is a Morse element of~$G$ satisfying $\inf(x^k)=0$ for every positive integer~$k$. For $i\in \mathbb N$, denote $\ell_i=\ell(x^i)$.
Let $h\in G$  
and suppose that there exists an integer $i\geqslant 0$ such that $x^i\preccurlyeq \underline h$. 
Consider the initial segment of the path $A(1,h)$ in~$\X$ of length~$\ell_i$ (the same length as $x^i$). 
Then this segment stays within distance $M_x$ from the axis of~$x$: 
$$ \text{if \ }0\leqslant k\leqslant \ell_i \text{ \ then \ } d_{\X}((\underline h\wedge \Delta^{k})\gpD, \mathrm{axis}(x)) \leqslant M_x$$
Moreover, 
$$d_{\X}((\underline h\wedge \Delta^{\ell_i})\gpD, \  x^i\gpD)\leqslant 2M_x$$
\end{prop}

\begin{figure}[htb]
\begin{center}
\includegraphics{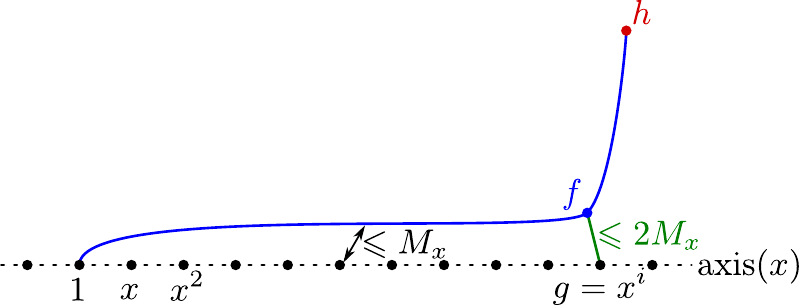}
\end{center}
\caption{The statement of Proposition~\ref{P:Matthieu'sLemma}. Here the segment from $\ast$ to~$f\gpD$, with $f=\underline h\wedge \Delta^{\ell_i}$, has the same length as the segment from $\ast$ to~$g\gpD$, namely $\ell_i$.}
\end{figure}

\begin{proof}
Let $f = \underline h \wedge \Delta^{\ell_i}$ and $g=x^i$. 
Of course, $\underline f =f$, and our first hypothesis on~$x$ says that $\underline g =g$. Because $x^i \preccurlyeq \underline h$, we have $1\preccurlyeq \underline g \preccurlyeq \underline f$ and by construction of $f$, the paths $A(1,f)$ and $A(1,g)$ have the same length.
By Lemma~\ref{L:ConcatQuasiGeod}(ii), the concatenation of paths $A(1,f)$, followed by $A(f,g)$ is a $(2,0)$-quasi-geodesic. This yields the first statement, recalling that $M_x=M_x^{(2,0)}$ is the Morse constant. 

For the second statement, we notice that in particular, the vertex $f\gpD$ is at distance at most~$M_x$ from 
some point $x^k\gpD$ on the axis. By the triangle inequality, $$\ell_i - M_x \leqslant d_\X(\ast,x^k\gpD) \leqslant \ell_i + M_x,$$ and therefore $x^k\gpD$ lies in an $M_x$-neighbourhood of~$g\gpD$. We conclude that 
\[d_\X (f\gpD,g\gpD) \leqslant d_\X (f\gpD,x^k\gpD) + d_\X (x^k\gpD,g\gpD) \leqslant 2M_x.\qedhere\]
\end{proof}


\section{Projection to the axis}\label{S:Projection}
In this section, we will define, in Garside-theoretical terms, a projection from~$\X$ to the axis of any element $x$ of $G$, provided that $\inf(x)=0$ and that $x$ satisfies the additional hypothesis of being \emph{right-rigid} (Definition~\ref{D:Rigid}).
If, moreover, $x$ is Morse, then this projection satisfies a contraction property (Proposition~\ref{P:GeodCloseToProj}) which extends Proposition~\ref{P:Matthieu'sLemma}.  This will be sufficient for deducing that our projection coincides, up to a bounded error, with any closest point projection (Corollary~\ref{C:CloseToClosest}). 

\begin{remark}\label{R:RightRigidity}
We remark that the occurrence of the condition of right-rigidity in our context is quite surprising, as this is a condition on the \emph{right} normal form of~$x$, whereas otherwise, we are generally using \emph{left} normal forms throughout this paper.
\end{remark}

\begin{lemma}\label{L:Projection0}
Let $x\in G$ be such that $\inf(x^k)=0$ for all $k\in \mathbb N$. Let $h\in G$. Then there exists $k_0\in \mathbb N$ such that 
for $k\geqslant k_0$, $x^{k-k_0}\preccurlyeq \underline{x^{k}h}$.
\end{lemma}

\begin{proof}

Consider the sequence $(\inf(x^k h))_{k\in \mathbb N}$; it is non-decreasing, as by \cite[Section~1]{ElRifaiMorton}, we have 
$\inf(x^{k+1}h)\geqslant\inf(x)+ \inf(x^kh)$. The same argument shows that the sequence is bounded above by $\sup(h)$, as
$$0 = \inf(x^k)\geqslant \inf(x^kh)+\inf(h^{-1}) = \inf(x^kh)-\sup(h).$$

We deduce that the sequence is eventually constant: there exist $k_0\in \mathbb N$ and $n_0\in \mathbb Z$, $n_0\leqslant \sup(h)$, such that for all $k\geqslant k_0$, 
$\inf(x^kh)=n_0$. 
Now, for $k\geqslant k_0$, $$\underline{ x^{k} h} = x^k h \Delta^{-n_0}= x^{k-k_0} x^{k_0}h\Delta^{-n_0}= x^{k-k_0} \underline{x^{k_0}h}$$
which finishes the proof. 
\end{proof}

It follows in particular that if $\inf(x^k)=0$ for all $k\in \N$, then for any $h\in G$, the set $\{k\in \mathbb Z, x\not\preccurlyeq \underline{x^kh}\}$ is bounded above.

\begin{definition}[Garside-theoretical projection to the axis of~$x$] 
Let $x\in G$ be such that $\inf(x^k)=0$ for all $k\in \N$.

(a) Let $h\in G$. We define the integer 
$$\lambda(h)= -\max\{k\in \Z, x\not\preccurlyeq \underline {x^kh}\}$$ and we note that for any $t\in\mathbb Z$, $\lambda(h\Delta^t)=\lambda(h)$. 

(b) We define a map $\pi\co G\to G$ as follows: for any $h\in G$, we set $\pi(h)=x^{\lambda(h)}$.
For any $t\in\mathbb Z$, we have $\pi(h\Delta^t)=\pi(h)$ so we can define the \emph{Garside-theoretical projection to the axis of $x$} to be the map (which we denote with the same letter) 
$$\pi\co \X\to \X, \ h\gpD \mapsto x^{\lambda(h)}\gpD.$$
\end{definition}

From now on, for the rest of the paper, we make the stronger hypothesis that $x$ satisfies $\inf(x)=0$ and is \emph{right-rigid} (see Definition~\ref{D:Rigid}). This implies in particular that $\inf(x^k)=0$ for all $k\in\mathbb N$.

\begin{lemma}\label{L:ProjectionAndPrefixes1}
Let $x\in G$ with $\inf(x) = 0$ be a \emph{right-rigid} element. Let $h\in G$. The following are equivalent:
\begin{itemize}
    \item[(i)] $x\preccurlyeq \underline h$,
    \item[(ii)] for all $k\geqslant 0$, $x^{k+1}\preccurlyeq \underline{x^kh}$.
\end{itemize}
\end{lemma}

\begin{proof}
Only (i) $\Longrightarrow$ (ii) needs a proof. 
Suppose that $x\preccurlyeq \underline h$, and let $k>0$. 
We claim that $\underline{x^kh}=x^k\underline{h}$ (i.e. that $\inf(x^k\underline{h})=0$). In order to prove this claim, we write $\underline{h}=xa$ with $a$ positive. 
Since $x$ is right-rigid, the condition $\inf(xa)=0$ implies $\inf(x^{k+1} a)=0$ (see \cite[Lemma~1]{CalvezWiestAH}). 
This means that $\inf(x^k \underline{h})=0$, proving the claim. 
Now (ii) is an immediate consequence.
\end{proof}

When $x$ is right-rigid, Lemma~\ref{L:ProjectionAndPrefixes1} yields a clean interpretation of the Garside-theoretical projection to the axis of~$x$ (see Figure~\ref{F:DefProjection}):

\begin{lemma}\label{L:ProjectionAndPrefixes2}
Let $x\in G$ with $\inf(x)=0$ be a right-rigid element. Let $h\in G$.
For $m\in \mathbb Z$, 
the following conditions are equivalent: 
\begin{itemize}
\item[(i)] $m=-\lambda(h)$,
    \item[(ii)]
    {\rm{(a)}}\   $x\not\preccurlyeq \underline{x^m h}$ and
    {\rm{(b)}}\  $x\preccurlyeq \underline{x^{m+1}h}$,
    \item[(iii)] 
  for every $k\geqslant 0$, 
      {\rm{(a)}}\
   $x^{k+1}\not\preccurlyeq \underline{x^{m+k}h}$ and
   {\rm{(b)}}\  $x^k \preccurlyeq \underline{x^{m+k}h}$.
\end{itemize}
In particular, whenever $\lambda(h)>0$, we have $x^{\lambda(h)}\preccurlyeq \underline h$. Also, for every $k\in \mathbb Z$, $\lambda(x^k)=k$. 
\end{lemma}

\begin{figure}[htb]
\begin{center}
\includegraphics{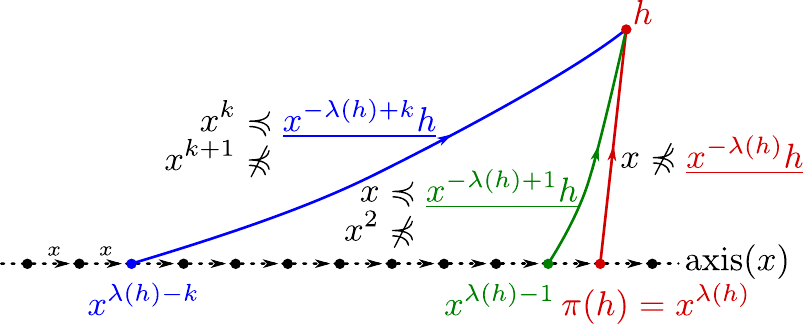}
\end{center}
\caption{The definition of the projection~$\pi$ from~$\X$ to the axis, for $h \in G$, and its properties (supposing that $x$ is right-rigid).}
\label{F:DefProjection}
\end{figure}

\begin{proof}
(i) $\Longrightarrow$ (ii) and (iii) $\Longrightarrow$ (i) are clear by definition of $\lambda$.  Assume (ii); this yields immediately statement (iii) for $k=0$. Suppose $k>0$. For (iii)(a), suppose on the contrary that $x^{k+1}\preccurlyeq \underline{x^{m+k}h}$. Then $\underline{x^{m+k} h}=x^{k+1}a$ for some positive~$a$ 
and $\underline{x^mh}=xa$, contradicting (ii)(a). For (iii)(b), we use the hypothesis (ii)(b) that $x\preccurlyeq \underline{x^{m+1}h}$; the conclusion then follows immediately from  Lemma~\ref{L:ProjectionAndPrefixes1}.
\end{proof}

\begin{lemma}\label{L:ProjectionLipschitz}
Let $x\in G$ with $\inf(x)=0$ be a right-rigid element.
Let $z\in G$ with $\inf(z)=0$, and let $s$ be a simple element. Suppose that $x\not\preccurlyeq z$. Then $x^2\not\preccurlyeq zs$. 
\end{lemma}

\begin{proof}
Let $x_r \cdots x_2\, x_1$ be the right normal form of~$x$.
Assume, contrary to the claim, that $x^{-2} zs$ is positive. 
Then by Lemma~\ref{L:CharneySplittings}(ii), $D_r(x^{-2}z)\preccurlyeq s$. 
We deduce (Lemma~\ref{L:CharneySplittings}(iii)) that $D_l(x^{-2}z)$ is also simple. 
On the other hand, by Lemma~\ref{L:CharneySplittings}(i), $x^2\succcurlyeq D_l(x^{-2}z)$
and, as $x$ is right-rigid, $x_1\succcurlyeq D_l(x^{-2}z)$. 
It then follows that 
$x x_r\cdots x_{2}\preccurlyeq  x^2 D_l(x^{-2}z)^{-1} \preccurlyeq z$, a contradiction. 
\end{proof}

\begin{proposition}[$\pi$ is $\ell(x)$-Lipschitz]\label{P:ProjectionLipschitz} 
Let $x\in G$ be a right-rigid element with $\inf(x)=0$
and canonical length $\ell=\ell(x)$. 
Suppose that $g,h\in G$ satisfy ${d_\X(g,h)=1}$. Then $|\lambda(g)-\lambda(h)|\leqslant 1$ and $d_\X(\pi(g),\pi(h))\leqslant \ell$. 
\end{proposition}

\begin{proof}
Let $k_g=-\lambda(g)$ and $k_h=-\lambda(h)$. 
Apply Lemma~\ref{L:AdjacencyInX} to the adjacent vertices $x^{k_g} g\gpD$ and $x^{k_g} h\gpD$: one of the following holds for some simple element~$s$: $(\underline{x^{k_g} g})s=\underline{x^{k_g}h}$ or $(\underline{x^{k_g}h})s=\underline{x^{k_g}g}$. 
Let us consider the first case. By definition of~$k_g$, 
$x\not\preccurlyeq \underline{x^{k_g}g}$ and by 
 Lemma~\ref{L:ProjectionLipschitz}, 
 $x^2\not\preccurlyeq (\underline{x^{k_g}g})s=\underline{x^{k_g}h}$. 
 From Lemma~\ref{L:ProjectionAndPrefixes1}, we deduce that $x\not\preccurlyeq \underline {x^{k_g-1} h}$. 
 This means that $k_g-1\leqslant k_h$. 
 In the second case, we obtain that $x\not\preccurlyeq \underline{x^{k_g} h}$, whence $k_g\leqslant k_h$. 
A similar reasoning applied to the adjacent vertices $x^{k_h}g\gpD$ and $x^{k_h}h\gpD$ shows that either $k_h-1\leqslant k_g$ or $k_h\leqslant k_g$. 
In either case, we obtain the desired claim for $\lambda(g)=-k_g$ and ${\lambda(h)=-k_h}$. 

Finally, the inequality $d_\X(\pi(g),\pi(h))\leqslant \ell$ is an immediate consequence.
\end{proof}

Let us now combine the rigidity and the Morse hypothesis on~$x$:

\begin{proposition}\label{P:GeodCloseToProj}
Let $x\in G$ be a right-rigid Morse element with $\inf(x)=0$ and canonical length~$\ell$. 
Let $\pi$ be the Garside-theoretical projection to ${\rm axis}(x)$. 
Then there exists a $D\in \mathbb N$ such that for any $h\in G$ and for any $i\in\mathbb Z$, the preferred geodesic $A(x^i,h)$ in~$\X$ passes at distance at most~$D$ from~$\pi(h\gpD)$: 
there exists some $h'\in G$ with $h'\gpD$ belonging to $A(x^i,h)$ such that
$$ d_\X(h',\pi(h))\leqslant D $$
Specifically, we can take $D=(\ell+1)\cdot M_x$.
\end{proposition}

\begin{figure}[htb]\label{F:GeodCloseToProj}
\begin{center}
\includegraphics{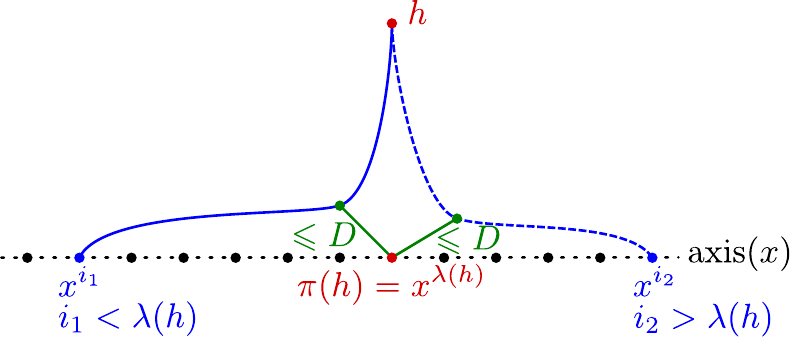}
\end{center}
\caption{The statement of Proposition~\ref{P:GeodCloseToProj}}
\end{figure}

\begin{proof}
We treat the cases $i<\lambda(h)$ and $i>\lambda(h)$ separately.\\

{\bf Case $i<\lambda(h)$ : } After the action of $x^{-i}$, 
we can assume without loss of generality that $i=0$. Noting that $x^{\lambda(h)}\preccurlyeq \underline h$,
we are then precisely in the situation of Proposition~\ref{P:Matthieu'sLemma}. Note that by the rigidity hypothesis, $\ell(x^{\lambda(h)})= \lambda(h)\cdot \ell$. Thus, 
if we define $h'=\underline h \wedge \Delta^{\lambda(h)\cdot\ell}$, we have 
$$ d_\X(h', \pi(h)) \leqslant 2\cdot M_x.$$

{\bf Case $i>\lambda(h)$ : } This time we will assume, again without loss of generality (after the action of $x^{-\lambda(h)+1}$), that $\lambda(h)=1$. 
Thus $\underline{\pi(h)}=x$ and $x\preccurlyeq \underline h$ but $x^2\not\preccurlyeq \underline h$ (Lemma~\ref{L:ProjectionAndPrefixes2}).

Let $h'=x^i \wedge \underline h$ and note that $\underline{h'}=h'$. 
We know from Lemma~\ref{P:PreferredPath}(i) that the vertex~$h'\gpD$ lies on the preferred geodesic $A(h,x^i)$, 
and our aim now is to bound its distance from~$\pi(h\gpD)$.

We make two observations about the vertex~$h'\gpD$. The first observation is that it lies in the $M_x$-neighbourhood of axis$(x)$, where we recall that $M_x$ is the Morse constant for $(2,0)$-quasi-geodesics with endpoints on axis$(x)$.
This follows from Lemma~\ref{L:ConcatQuasiGeod}(i) and the fact that $x\preccurlyeq \underline h'\preccurlyeq x^i$. 

The second observation about the vertex~$h'\gpD$ is that it has the same projection to the axis as $h\gpD$:
$$\pi(h')=\pi(h)$$
Here is a proof of this fact. We have to prove that $\lambda(x^i\wedge\underline h) = \lambda(h)=1$. By Lemma~\ref{L:ProjectionAndPrefixes2}, it suffices to prove that
$$ x\preccurlyeq x^i \wedge \underline h \text{ \ \ but \ \ } x^2 \not\preccurlyeq x^i\wedge \underline h.$$
Keeping in mind the hypothesis that $i\geqslant 2$,
this follows immediately from the analogous condition on $\underline h$.
This completes the proof of the second observation about~$h'\gpD$.

By the first observation, there exists an integer~$k$ with 
$$d_\X(h',x^k)\leqslant M_x$$
Since the projection~$\pi$ is $\ell$-Lipschitz (Proposition~\ref{P:ProjectionLipschitz}), we have
$$ d_\X(\pi(h'),x^k) = d_\X(\pi(h'),\pi(x^k)) \leqslant \ell\cdot M_x$$
Applying the triangle-inequality we obtain
$$ d_\X(h',\pi(h')) \leqslant \ell \cdot M_x + M_x = (\ell+1)\cdot M_x $$
Also, by the second observation above we have $\pi(h')=\pi(h)$, so
$$ d_\X(h',\pi(h)) \leqslant (\ell+1)\cdot M_x $$
The proof of Proposition~\ref{P:GeodCloseToProj} is complete, with 
\[ D = \max \left( 2\cdot M_x,(\ell+1)\cdot M_x\phantom{^!}\right) = (\ell+1)\cdot M_x \qedhere\]
\end{proof}

We deduce that $\pi$ is uniformly close to the closest point projection:
\begin{corollary}\label{C:CloseToClosest}
Let $x\in G$ with $\inf(x)=0$ be a right-rigid Morse element, and $\pi$ be the Garside-theoretical projection to ${\rm axis}(x)$. 
Let $h\in G$ and let $x^k\gpD$ be any point of the axis such that 
$$
d_\X(h,x^k) = \min_{i\in\mathbb Z} d_\X(h,x^i)
$$
Then $$d_\X(\pi(h),x^k) \leqslant 2D$$
where $D$ is the constant promised by Proposition~\ref{P:GeodCloseToProj}.
\end{corollary}

\begin{figure}[htb]\label{F:TwoProjectionsClose}
\begin{center}
\includegraphics{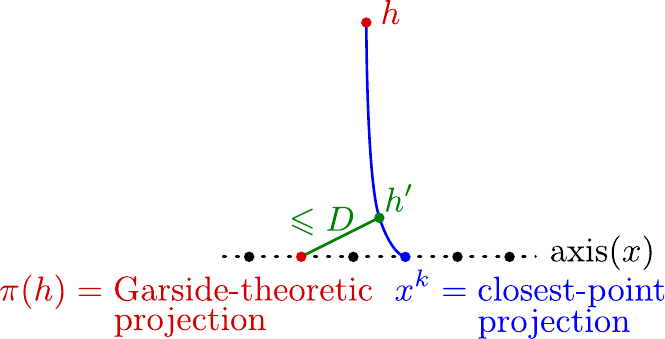}
\end{center}
\caption{The projection $\pi$ is uniformly close to any closest-point projection}
\end{figure}

\begin{proof}
By Proposition~\ref{P:GeodCloseToProj}, there is a point~$h'\gpD$ on the geodesic $A(h,x^k)$ such that $d_\X(\pi(h),h')\leqslant D$. 
Since $x^k\gpD$ is a point on the axis as close as possible to~$h\gpD$, we must have
$$ d_\X(x^k,h') \leqslant d_\X(\pi(h),h') \leqslant D $$
By the triangle inequality, we conclude
\[ d_\X(\pi(h), x^k) \leqslant 2D \qedhere\]
\end{proof}


\section{The strong contraction property}\label{S:StrongContraction}

In this section we recall the definition of the strong contraction property and the strong constriction property.
Then we prove the main result of this paper: in a $\Delta$-pure Garside group of finite type~$G$, the axis of any Morse element is strongly contracting.

The following two definitions and a proof of their equivalence can be found in~\cite{ArzhCashenTao}. 

\begin{definition}\label{D:StronglyContracting}
Let $(X,d)$ be a metric space and let $\mathcal A$ be any subset of $X$. A map $\rho:X\to \mathcal A$ is \emph{$C$-strongly contracting} for $C\geqslant 0$ if the following hold:

\begin{itemize}
\item[(i)]$\rho$ is coarsely equivalent to $id_{\mathcal A}$ on $\mathcal A$:
for every $a\in \mathcal A$, $d(\rho(a),a)\leqslant C$,
\item[(ii)]
for all $x\in X$, $d(x,\rho(x))-d(x,\mathcal A)\leqslant C$,
\item[(iii)] for all $u,v\in X$, $d(u,v)<d(v,\mathcal A)-C$ implies that $d(\rho(u),\rho(v))\leqslant C$ (i.e.\ if a ball in~$\X$ is disjoint from~a $C$-neighborhood of $\mathcal A$, then its image under $\rho$ is contained in a ball of radius~$C$).
\end{itemize}
The map $\rho$ is \emph{strongly contracting} if 
there exists a non-negative integer~$C$ such that $\rho$ is $C$-strongly contracting. The subset $\mathcal A\subset X$ is \emph{strongly contracting} if there exists a strongly contracting map $X\to \mathcal A$.
\end{definition}

Note (\cite[Lemma 2.8]{ArzhCashenTao}) that a strongly contracting map ${\rho\co X\to \mathcal A}$ satisfies in fact a strengthened version of clause (ii) in Definition~\ref{D:StronglyContracting}: namely, $\rho$ is 
\emph{coarsely a closest point-projection} to $\mathcal A$, meaning that for all $x\in X$, there exists an $a\in \mathcal A$ with $d(x,\mathcal A) =d(x,a)$ such that $d(\rho(x),a)$ is uniformly bounded.

As proven in \cite[Proposition 2.9]{ArzhCashenTao}, a map $\rho$ is strongly contracting if and only if it is \emph{strongly constricting}; this alternative characterization will be useful in Lemma~\ref{L:Hausdorff} and Section~\ref{S:LoxOnCAL}:

\begin{definition}\label{D:StronglyConstricting}
Let $(X,d)$ be a metric space and let $\mathcal A$ be any subset of $X$. A map $\rho:X\to \mathcal A$ is \emph{$C$-strongly constricting} for $C\geqslant 0$ if the following hold:
\begin{itemize}
    \item[(i)] $\rho$ is coarsely equivalent to $id_{\mathcal A}$ on $\mathcal A$: for every $a\in \mathcal A$, $d(\rho(a),a)\leqslant C$,
    \item[(ii)]
    for every geodesic~$\gamma$ in~$X$ with endpoints $x_0$ and~$x_1$, if $d(\rho(x_0), \rho(x_1))>C$, then $d(\rho(x_i),\gamma)<C$ for $i=0,1$. 
\end{itemize}
The map $\rho$ is \emph{strongly constricting} if there exists a non-negative integer~$C$ such that~$\rho$ is $C$-strongly constricting.
\end{definition}

\begin{lemma}\label{L:Hausdorff}
Let $(X,d)$ be a metric space and let $\mathcal A$ be a subset of $X$. Let $\rho\co X\to \mathcal A$  be a  strongly contracting map. 
\begin{enumerate}
    \item[(i)] Suppose that $\mathcal B\subset X$ is another subset of~$X$ with $d_{\rm Hausdorff}(\mathcal A,\mathcal B)<\infty$. Then there is a strongly contracting map $\rho'\co X\to \mathcal B$.
    \item[(ii)] Let $(X',d')$ be another metric space. Suppose there is an isometric and quasi-surjective embedding $\iota\co X\hookrightarrow X'$. 
    Then there is a strongly contracting map $\rho':X'\to \iota(\mathcal A)$.
\end{enumerate}
\end{lemma}

\begin{proof}
For (i), let $\delta=d_{\rm Hausdorff}(\mathcal A,\mathcal B)$. We construct $\rho'$ by choosing, for any $x\in X$, a point $b\in \mathcal B$ with $d(b,\rho(x))\leqslant \delta$, and declaring that $\rho'(x)=b$. Thus $\rho$ and $\rho'$ are $\delta$-coarsely equivalent. Now it is an easy exercise to show that if $\rho$ is $C$-strongly constricting then $\rho'$ is $(C+2\delta)$ strongly constricting.

For (ii), let $\varepsilon$ be such that the $\varepsilon$-neighbourhood of~$\iota(X)$ in~$X'$ is all of~$X'$. We define $\rho'$ by choosing, for every $x'\in X'$, a point $x\in X$ with $d'(\iota(x),x')\leqslant \varepsilon$, and declaring that $\rho'(x')=\iota(\rho(x))$. We have to prove that $\rho'$ is strongly contracting. More precisely, supposing that $\rho\co X\to \mathcal A$ is $C$-strongly contracting, our aim is to prove that $\rho'$ is $(C+3\epsilon)$-strongly contracting. We shall prove only part (iii) of Definition~\ref{D:StronglyContracting}; the other two clauses can be checked easily.

For any point $v'\in X'$, consider the ball~$B'$ centered in~$v'$ and of radius $d'(v',\mathcal A)-C-3\varepsilon$. If we choose a point of $\iota(X)$ at distance at most~$\varepsilon$ from each point of~$B'$, we obtain a subset of~$X$ which is contained in a ball in~$X$ centered at some point $v$ (with $d'(\iota(v),v')\leqslant\varepsilon$) and of radius $d'(v',\mathcal A)-C-\varepsilon$. Since $d'(v',\mathcal A)-C-\varepsilon \leqslant d(v,\mathcal A)-C$, the projection $\rho'(B')$ is contained in $\iota(\rho(B))$, where $B$ is the ball in~$X$ centered in~$v$ and of radius $d(v,\mathcal A)-C$. By hypothesis, ${\rm diam}(\rho(B))\leqslant C < C+3\varepsilon$, which is what we wanted to prove.
\end{proof}

\begin{proposition}\label{P:StrongContr}
Let $G$ be a $\Delta$-pure Garside group of finite type. Let $x\in G$ with $\inf(x)=0$ be a right-rigid Morse element. 
The Garside-theoretical projection $\pi$ to ${\rm axis}(x)$
is $5D$-strongly contracting, where $D$ is the constant promised by Proposition~\ref{P:GeodCloseToProj}.
\end{proposition}

\begin{proof}
The first and second conditions of Definition~\ref{D:StronglyContracting} follow respectively from 
the last statement of Lemma~\ref{L:ProjectionAndPrefixes2} and Corollary~\ref{C:CloseToClosest}, which asserts the stronger condition that $\pi$ is coarsely a closest-point projection. 
Let us prove that condition~(iii) is satisfied.  Let $h,g\in G$ be such that
$d_{\X}(h,g)\leqslant d_{\mathcal X}(h\gpD,{\rm axis}(x))$ (that is, $g\gpD$ lies in a ball in $\X$ centered at~$h\gpD$ and disjoint from the axis of $x$).

Let us denote $r_*=d_\X(h,\pi(h))$. 
Now, 
$$d_\X(h,g) \leqslant d_\X(h\gpD,{\rm axis}(x)) \leqslant r_*.$$
By Proposition~\ref{P:GeodCloseToProj}, the preferred geodesic $A(g,\pi(h))$ contains a point~$g'\gpD$ at distance at most~$D$ from~$\pi(g\gpD)$.
By Lemma~\ref{P:PreferredPath}(v) (convexity of balls), we have $d_\X(h,g')\leqslant r_*$.

\begin{figure}[htb]
\begin{center}
\includegraphics{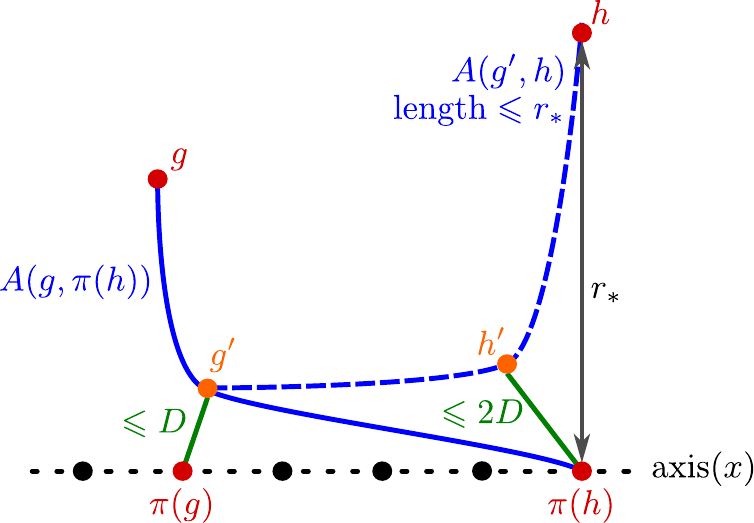}
\end{center}
\caption{The proof of Proposition~\ref{P:StrongContr}}
\end{figure}

Let us now study the preferred geodesic $A(h,g')$. We have just seen that it is of length at most~$r_*$. 
Moreover, by Proposition~\ref{L:AFellowTraveller}, it is at Hausdorff distance at most~$D$ from~$A(h,\pi(g))$, 
which in turn passes at distance at most~$D$ from~$\pi(h\gpD)$ (by Proposition~\ref{P:GeodCloseToProj} again). Thus~$A(h,g')$ contains a point~$h'\gpD$ at distance at most $2D$ from~$\pi(h\gpD)$. 

Now $d_\X(h',h)\geqslant r_*-2D$ by the triangle inequality. Therefore $$d_\X(g',h') = d_\X(g',h) - d_\X(h',h) \leqslant 2D$$
and we obtain the desired conclusion:
\begin{align*}
    d_\X(\pi(g), \pi(h)) & \leqslant  d_\X(\pi(g),g') + d_\X(g',h') + d_\X(h',\pi(h))\\
      & \leqslant D + 2D + 2D = 5D \qedhere 
\end{align*}
\end{proof}

The following is the main result of this paper:

\begin{theorem}[Strong contraction property of axes]\label{T:StrongContr}
Let $G$ be a $\Delta$-pure Garside group of finite type. Let $g\in G$ be a Morse element. Then
\begin{enumerate}
    \item[(i)] in $\X=\Gamma(G)/\gpD$, the axis $\{g^k\gpD\ | \ k\in\mathbb Z\}\subset \X$ is strongly contracting.
    \item[(ii)] in $\overline \Gamma = \Gamma(G/Z(G))$, the axis $\{g^kZ(G) \ | \ k\in\mathbb Z\}\subset \overline \Gamma$ is strongly contracting.
\end{enumerate}
\end{theorem}

\begin{proof}[Proof of Theorem~\ref{T:StrongContr}]
First we recall that the axis of~$g$ being Morse in~$\X$ or in $\overline \Gamma$ are equivalent properties, because the property of being Morse is invariant under quasi-isometry.

Now, by Proposition~\ref{P:PowerConjToRigid}, there is an element $x\in G$ with $\inf(x)=0$ which is right-rigid, and which is obtained from~$g$ by taking a power, conjugating by some element $a\in G$, and multiplying by a central element. 
Thus in both spaces, $\X$ and~$\overline \Gamma$, taking the axis of~$x$ and translating it by the action of~$a$ yields a subset which is at finite Hausdorff-distance from the axis of~$g$. 
By Proposition~\ref{P:StrongContr}, the axis of $x$ is strongly contracting in $\X$,
and so is its image under the $a$-action; by Lemma~\ref{L:Hausdorff}(i), the axis of $g$ is strongly contracting in $\X$. 

For Theorem~\ref{T:StrongContr}(ii) we recall from Proposition~\ref{P:XandGammaBarAreQI} that there is an isometric embedding $\iota\co \X \hookrightarrow \overline \Gamma$ with $\lfloor\frac{e}{2}\rfloor$-dense image. 
The vertices of the image are those which are represented by elements $g$ with $\inf(g)\equiv 0 \mod e$. 
In particular, the axis of~$x$ in~$\overline \Gamma$ is
the image under $\iota$ of the axis of $x$ in $\mathcal X$.  
By Proposition~\ref{P:StrongContr} and Lemma~\ref{L:Hausdorff}(ii), the axis of $x$ in $\overline{\Gamma}$ is strongly contracting (and so is its $a$-translate). By Lemma~\ref{L:Hausdorff}(i), the axis of~$g$ in~$\overline \Gamma$ is also strongly contracting.
\end{proof}

\begin{remark}
The proof did not use the full strength of the Morse hypothesis. 
We only ever considered (1,0)-quasi-geodesics (i.e.\ actual geodesics) and $(2,0)$-quasi-geodesics consisting of two geodesic segments: we only needed these two types of paths to stay at bounded distance from the axis. We showed that for $\Delta$-pure Garside groups of finite type, this condition on axes is equivalent to both strong contractibility and to the full Morse property.
\end{remark}


\section{Consequences for the additional length graph}\label{S:LoxOnCAL}

In this section, we record a consequence of Theorem~\ref{T:StrongContr} for the study of the \emph{additional length graph} $\CAL(G)$ of a Garside group~$G$.
This graph was introduced in~\cite{CalvezWiestCurve}, and further studied in~\cite{CalvezWiestAH} and~\cite{CalvezWiestSurvey} (see also~\cite{CalvezEuclidean}). We briefly recall the definition and the main results from~\cite{CalvezWiestCurve}:

\begin{definition}\label{D:CAL} Let $G$ be a Garside group of finite type.

(a) An element $h\in G$ is \emph{absorbable} if it satisfies two conditions:
\begin{enumerate}
    \item[(i)] $\inf(h)=0$ \ or \ $\sup(h)=0$
    \item[(ii)] there exists an element $g\in G$ which ``absorbs'' $h$, meaning that $\inf(g)=\inf(gh)$ and $\sup(g)=\sup(gh)$. 
\end{enumerate}

(b) The \emph{additional length graph} $\CAL(G)$ of~$G$ is the (usually locally-infinite) graph with the same set of vertices and edges as~$\mathcal X$, but with, additionally, a new edge between vertices $g\gpD$ and~$h\gpD$ whenever there is an absorbable element $s\in G$ so that $\underline g s\in h\gpD$. 
The graph metric of $\CAL(G)$ is denoted by~$\dAL$: for vertices $g\gpD$, $h\gpD$ of $\CAL$, we sometimes denote $\dAL(g,h)=\dAL(g\gpD,h\gpD)$.
\end{definition}

\begin{proposition}\cite[Lemmas 1-3]{CalvezWiestCurve}\label{P:Absorbable}
\begin{itemize}
    \item[(i)] An element $h\in G$ is absorbable if and only if $h^{-1}$ is.
    \item[(ii)] If $h=h_1\cdot h_2\cdot h_3$, with $\inf(h)=$\ $\inf(h_1)=$ $\inf(h_2)=\inf(h_3)=0$ is absorbable, 
    then $h_1, h_2$ and $h_3$ are also absorbable.
     \item[(iii)] Suppose that $h\in G$ is absorbable; then there exists an absorbing element $g$ with $\inf(g)=0$ and $\sup(g)=\ell(h)$. 
\end{itemize}
\end{proposition}

Since there is a natural inclusion $\X\hookrightarrow \CAL(G)$, we can interpret the family of paths $A(g,h)$ from Definition~\ref{D:PreferredPaths} as a family of paths in~$\CAL(G)$.

\begin{proposition}[Properties of $\CAL$] \cite[Theorem 1]{CalvezWiestCurve}\label{P:CALproperties}
Let $G$ be a Garside group of finite type. 
\begin{enumerate}
    \item[(i)] The additional length graph $\CAL(G)$ is $60$-hyperbolic.
    \item[(ii)] The paths $A(g,h)$ form a uniform family of unparameterized quasi-geodesics in $\CAL(G)$.
\end{enumerate}
\end{proposition}

\begin{remark}\label{R:AbsorbableComments}
\begin{enumerate}
    \item[(a)] If $G$ is the braid group~$B_n$, equipped with the classical or dual Garside structure, then $\CAL(G)$ is conjectured to be quasi-isometric to the curve graph of the $(n+1)$-times punctured sphere.
    \item[(b)] Note that Proposition~\ref{P:CALproperties} is \emph{not} claiming that ${\rm diam}(\CAL(G))=\infty$. 
    For instance, the group $G=\mathbb Z^3$ carries a Garside structure with~$\Delta=(1,1,1)$ (see~\cite[Chapter~1,1.1]{DDGKM}), for which all elements $(k,0,0), (0,k,0)$ and $(0,0,k)$ (with $k\in\mathbb Z$) are absorbable, so that ${\rm diam}(\CAL(\Z^3))=3$. 
    By contrast, for any Artin group of spherical type~$A$ we do have ${\rm diam}(\CAL(A))=\infty$; 
    the proof of this fact in~\cite{CalvezWiestAH} involved an explicit Garside-theoretical construction of elements with (very) strongly constricting axes.
\end{enumerate}
\end{remark}

\begin{theorem}\label{T:MorseLoxOnCAL}
For a $\Delta$-pure Garside group of finite type~$G$ we consider the additional length graph~$\CAL(G)$, equipped with the $G/Z(G)$-action.
For any Morse element~$g$, the action of~$g$ on $\mathcal{C}_{AL}(G)$ is loxodromic and WPD.
\end{theorem}

Here WPD is the weak proper discontinuity condition of~\cite{BestvinaFujiwara}: 

\begin{definition}
The action of $g\in G/Z(G)$ on $\CAL(G)$ is \emph{weakly properly discontinuous} (WPD) if 
for every (equivalently, for any) $k\in G$, and for every $\kappa>0$, there exists $N>0$ such that for all $n\geqslant N$, the set
$$\{ h\in G/Z(G),\ \dAL(k,hk)\leqslant \kappa,\ \dAL(g^Nk, hg^nk)\leqslant \kappa\}$$ is finite. 
\end{definition}

\begin{corollary}\label{C:MorseDiamInf}
If $G$ contains a Morse element then ${\rm diam}(\CAL(G))=\infty$.
\end{corollary}

\begin{corollary}
Pseudo-Anosov braids act loxodromically and WPD on $\CAL(B_n)$.
\end{corollary}

\begin{proof}[Proof of Theorem~\ref{T:MorseLoxOnCAL}]
Let $g$ a Morse element of $G$; by Proposition~\ref{P:PowerConjToRigid}, $g$ has a power which is conjugate to a right-rigid element of the form $\Delta^{em}x$ with $x$ right-rigid and $\inf(x)=0$. Thus it suffices to prove the theorem for a right-rigid Morse element~$x$ with $\inf(x)=0$. 

We know from Proposition~\ref{P:StrongContr} and from \cite[Proposition 2.9]{ArzhCashenTao} that there is a constant $C\in\mathbb N$ such that the Garside-theoretical projection $\pi\co \mathcal X \to {\rm axis}(x)$ is  $C$-strongly constricting.

\begin{lemma}\label{L:AbsorbableShortProj}
There is a constant $F\in\mathbb N$ with the following property:
suppose that we have $h_1,h_2\in G$ and an absorbable element $s\in G$ such that $\underline{h_1}s\in h_2\gpD$.
Then $$ d_\X(\pi(h_1),\pi(h_2)) < F.$$
\end{lemma}

\begin{proof}[Proof of Lemma~\ref{L:AbsorbableShortProj}]
After exchanging the roles of $h_1$ and $h_2$, if necessary, we can suppose that $\inf(s)=0$ (rather than $\sup(s)=0$). 
We are going to prove that the bound $F=2M_x^{(2,C)}+6 C$ works, where $M_x^{(2,C)}$ is the Morse constant for $(2,C)$-quasi-geodesics with endpoints on ${\rm axis}(x)$.

If $d_\X(\pi(h_1),\pi(h_2))\leqslant C$, then we are done.
If $d_\X(\pi(h_1),\pi(h_2))>C$, then the $C$-strong constriction property implies that the Garside normal form of~$s$ (as a word in the letters~$\mathcal D$) can be cut into three pieces, 
yielding a factorization $s=s_1\cdot s_2\cdot s_3$ with $\inf(s_1)=\inf(s_2)=\inf(s_3)=0$, and such that (see Figure~\ref{F:Absorbable})
\begin{figure}[htb]\label{F:Absorbable}
\begin{center}
\includegraphics{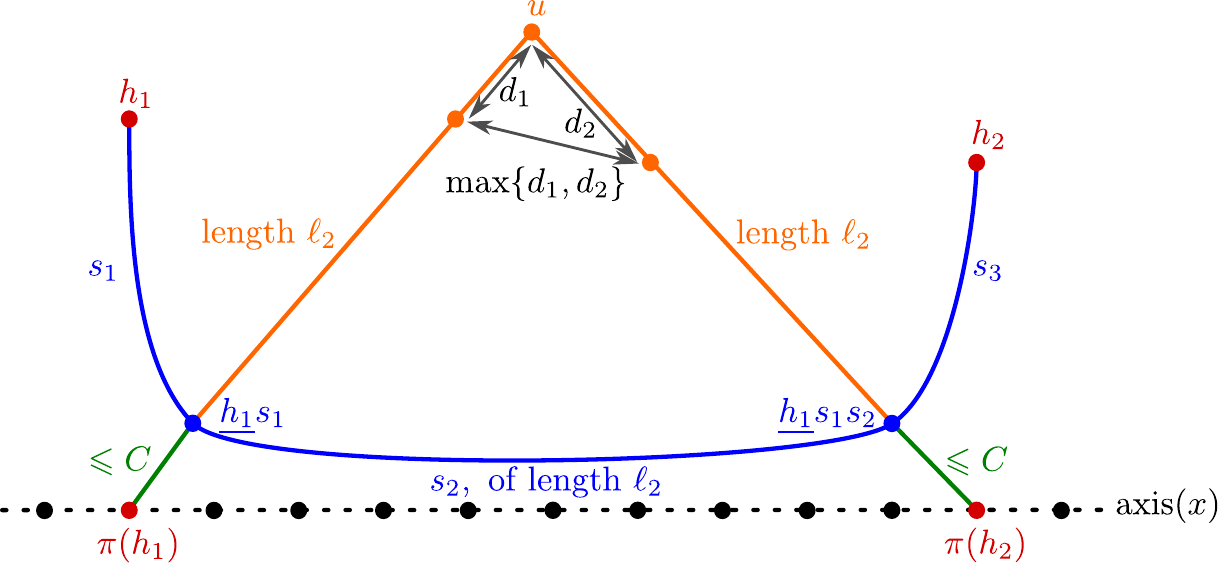}
\end{center}
\caption{The proof of Lemma~\ref{L:AbsorbableShortProj}}
\end{figure}
$$ d_\X( \underline{h_1}s_1, \pi(h_1) ) \leqslant C \text{ \ \ \ and \ \ \ } 
d_\X( \underline{h_1} s_1 s_2), \pi(h_2) ) \leqslant C.$$
By Proposition~\ref{P:Absorbable}(ii), all three factors $s_1$, $s_2$ and $s_3$ are absorbable. In particular,~$s_2$ is. Let us denote $\ell_2$ the Garside length of~$s_2$ -- thus $\inf(s_2)=0$, $\sup(s_2)=\ell_2$.

As seen in~\cite{CalvezWiestSurvey}, absorbability of~$s_2$ means that there is a geodesic triangle in~$\X$ which is equilateral of side length~$\ell_2$, and one of whose sides is the geodesic $A(\underline{h_1} s_1, \underline{h_1}s_1 s_2)$.
Moreover, for any two points in two different sides of this triangle, with distances $d_1$ and~$d_2$ from the shared corner of the triangle, the distance of the two points in~$\X$ is $\max(d_1,d_2)$. In particular, the triangle is $(2,0)$-quasi-isometrically embedded in~$\X$ (compare Lemma~\ref{L:ConcatQuasiGeod}).

Let $u\in G$ so that $u\gpD$ is the corner 
of the triangle furthest from the axis of~$x$.
We claim that the distance of $u\gpD$ from the axis is at least $\frac{l_2}{2}-2C$. 
Indeed, $$\ell_2-C\leqslant d_\X(u,\pi(h_i))\leqslant \ell_2+C$$ for $i=1,2$.
Moreover, for any $k\in\Z$ (not necessarily between $\lambda(h_1)$ and $\lambda(h_2)$) we have by the triangle inequality
\begin{multline*}
d_X(u,x^k)\geqslant \max\left( d_\X(x^k,\pi(h_2))-\ell_2-C, \ \ell_2-C-d_\X(x^k,\pi(h_1)),\right.\\
\left.\ell_2-C-d_\X(x^k,\pi(h_2)), \  d_\X(x^k,\pi(h_1))-\ell_2-C\right)
\end{multline*}

Using the fact that $\ell_2-2C\leqslant d_\X(\pi(h_1),\pi(h_2))\leqslant \ell_2+2C$, one can calculate that, depending on $k$, one of these four values is always at least $\frac{\ell_2}{2}-2C$. This completes the proof of the claim.

\begin{figure}[htb]
\begin{center}
\def\svgwidth{12cm}
\begingroup%
  \makeatletter%
  \providecommand\color[2][]{%
    \errmessage{(Inkscape) Color is used for the text in Inkscape, but the package 'color.sty' is not loaded}%
    \renewcommand\color[2][]{}%
  }%
  \providecommand\transparent[1]{%
    \errmessage{(Inkscape) Transparency is used (non-zero) for the text in Inkscape, but the package 'transparent.sty' is not loaded}%
    \renewcommand\transparent[1]{}%
  }%
  \providecommand\rotatebox[2]{#2}%
  \newcommand*\fsize{\dimexpr\f@size pt\relax}%
  \newcommand*\lineheight[1]{\fontsize{\fsize}{#1\fsize}\selectfont}%
  \ifx\svgwidth\undefined%
    \setlength{\unitlength}{215.01412525bp}%
    \ifx\svgscale\undefined%
      \relax%
    \else%
      \setlength{\unitlength}{\unitlength * \real{\svgscale}}%
    \fi%
  \else%
    \setlength{\unitlength}{\svgwidth}%
  \fi%
  \global\let\svgwidth\undefined%
  \global\let\svgscale\undefined%
  \makeatother%
  \begin{picture}(1,0.22311743)%
    \lineheight{1}%
    \setlength\tabcolsep{0pt}%
    \put(0,0){\includegraphics[width=\unitlength,page=1]{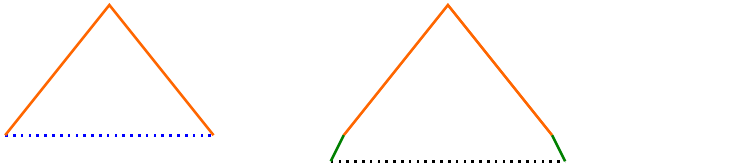}}%
    \put(0.38795791,0.01427323){\color[rgb]{0,0.50196078,0}\makebox(0,0)[lt]{\lineheight{1.25}\smash{\begin{tabular}[t]{l}$\leqslant C$\end{tabular}}}}%
    \put(0.75770142,0.01427323){\color[rgb]{0,0.50196078,0}\makebox(0,0)[lt]{\lineheight{1.25}\smash{\begin{tabular}[t]{l}$\leqslant C$\end{tabular}}}}%
    \put(0.09781051,0.11398708){\color[rgb]{0,0,0}\makebox(0,0)[lt]{\lineheight{1.25}\smash{\begin{tabular}[t]{l}$(2,0)$-\end{tabular}}}}%
    \put(0.05595276,0.08608191){\color[rgb]{0,0,0}\makebox(0,0)[lt]{\lineheight{1.25}\smash{\begin{tabular}[t]{l}quasi-geodesic\end{tabular}}}}%
    \put(0.55126948,0.0930582){\color[rgb]{0,0,0}\makebox(0,0)[lt]{\lineheight{1.25}\smash{\begin{tabular}[t]{l}$(2,C)$-\end{tabular}}}}%
    \put(0.50941172,0.06515323){\color[rgb]{0,0,0}\makebox(0,0)[lt]{\lineheight{1.25}\smash{\begin{tabular}[t]{l}quasi-geodesic\end{tabular}}}}%
    \put(0,0){\includegraphics[width=\unitlength,page=2]{TwoTriangles.pdf}}%
  \end{picture}%
\endgroup%

\end{center}
\caption{Left: the unit-speed parametrization of this path is a $(2,0)$-quasi-geodesic. \ Right: $\gamma$, which coincides with the previous path except for jumps of size at most $C$ at the starting and end point is a $(2,C)$-quasi-geodesic.}
\end{figure}

Now consider the path $\gamma\co [0,2\ell_2] \to \X$ 
\begin{itemize}
    \item with $\gamma(0)=\pi(h_1\gpD)$
    \item which for $t\in ]0,\ell_2]$ follows a unit speed parametrization of~$A( \underline{h_1}s_1, u)$ 
    \item which for $t\in [\ell_2,2\ell_2[$ follows a unit speed parametrization of~$A( u, \underline{h_1}s_1 s_2)$, and
    \item with $\gamma(2 \ell_2)=\pi(h_2\gpD)$
\end{itemize}

We see that $\gamma$ is a $(2,C)$-quasi-geodesic (because, apart from the jumps of size at most~$C$ at the starting and end points it is a $(2,0)$-quasi-geodesic).

Thus the Morse condition for the axis of~$x$ implies that $\frac{\ell_2}{2}-2C\leqslant M_x^{(2,C)}$, 
or equivalently, $\ell_2\leqslant 2M_x^{(2,C)}+4C$. Therefore
\[ d_\X( \pi(h_1) , \pi(h_2) ) \leqslant \ell_2+2C \leqslant 2M_x^{(2,C)}+6 C \qedhere \]
\end{proof} 

Coming back to the proof of Theorem~\ref{T:MorseLoxOnCAL}, suppose that the action of~$x$ is not loxodromic. This means that $t_n:=\dAL(1,x^n)$ grows sublinearly with $n$. 
Consider elements $1=h_0,h_1,\ldots, h_{t_n}=x^n$ of~$G$ and a geodesic in $\CAL(G)$ between $\ast$ and $x^n\gpD$ through the vertices $h_i\gpD$. 
The sublinear growth means that for sufficiently large values of~$n$, there must be an $i\in \{1,2,\ldots,t_n\}$ such that
$$d_{\mathcal X} (h_{i-1},h_i)\geqslant \max\{F,\ell(x)\}.$$ 
This contradicts either Lemma~\ref{L:AbsorbableShortProj} or
the Lipschitz property of~$\pi$ (Proposition~\ref{P:ProjectionLipschitz}). 
This completes the proof that the $x$-action on $\CAL(G)$ is loxodromic.

We now turn to the proof that the $x$-action on $\CAL(G)$ is WPD. Fix $\kappa>0$. 
Define $$ S_x^{(\kappa,n)} = \{ h\in G/Z(G) \ | \ \dAL(1,h)<\kappa \text{\ and\ } \dAL(x^n,h x^n)<\kappa \}$$

We look at the situation in~$\X$: denoting $E=\max\{F,\ell(x)\}$, 
Lemma~\ref{L:AbsorbableShortProj} and Proposition~\ref{P:ProjectionLipschitz} tell us that that for $h\in S_g^{(\kappa,n)}$,
$$
d_{\X}(1,\pi(h)) \leqslant E\cdot \kappa \text{ \ \ and \ \ } d_{\X}(x^n,\pi(hx^n)) \leqslant E\cdot \kappa
$$
We now choose $N$ sufficiently large so that $d_\X(1,x^N) > C + 2\cdot E\cdot \kappa$ -- then for all integers $n$ with $n\geqslant N$ we also have $d_\X(1,x^n) > C + 2\cdot E\cdot \kappa$, and by the triangle inequality
$$ d_\X(\pi(h),\pi(hx^n)) > C.$$

\begin{figure}[htb]
    \centering
    \includegraphics{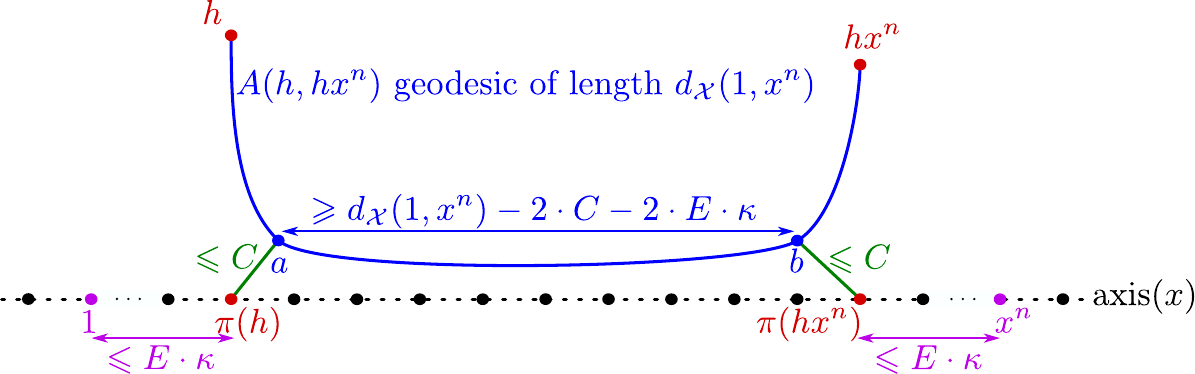}
    \caption{The proof that the action of $x$ is WPD.}
    \label{F:ProofWPD}
\end{figure}

The strong constriction property of~$\pi$ then guarantees that the geodesic $A(h,hx^n)$ passes through points $a\gpD$ and $b\gpD$ at distance at most~$C$ from $\pi(h\gpD)$ and $\pi(hx^n\gpD)$ respectively, and hence at distance at most $C+E\cdot \kappa$ from $\ast$ and $x^n\gpD$ respectively. 
Therefore we have $d_{\X}(a,b)\geqslant d_{\X}(1,x^n)-2\cdot C-2\cdot E\cdot\kappa$.

On the other hand, the geodesic $A(h,hx^n)$ has the same length as the segment of the axis $A(1,x^n)$ (as it is its image under left-translation by $h$).
Thus we have $$d_{\X}(h,a)=d_{\X}(1,x^n)-d_{\X}(a,b)-d_{\X}(b,hx^n)\leqslant d_{\X}(1,x^n)-d_{\X}(a,b) \leqslant 2\cdot C+2\cdot E\cdot \kappa.$$

We conclude that $d_\X(1,h) \leqslant 3\cdot (C+E\cdot \kappa)$. There are only finitely many elements $h\in G/Z(G)$ with this property. 
This completes the proof that the action of~$g$ is WPD.
In fact we have proven something slightly stronger than what was required: we found a bound on the size of the set $S_g^{(r,n)}$ which does not depend on~$n$, as long as $n\geqslant N$.
\end{proof} 




{\bf Acknowledgements.} The authors thank Yvon Verberne and Alessandro Sisto for useful discussions. 
The first author was partially supported by the EPSRC New Investigator Award EP/S010963/1. The first author acknowledges support by  FONDECYT 1180335.


\begin{thebibliography}{00}

\bibitem{ABD} {\bf Carolyn Abbott, Jason Behrstock, Matthew Gentry Durham}, {\it Largest acylindrical actions and stability in hierarchically hyperbolic groups}, Trans.\ Am.\ Math.\ Soc., Ser.\ B 8, 66--104 (2021). 

\bibitem{Adyan} {\bf Sergei Ivanovich Adyan}, {\it Fragments of the word $\delta$ in a Braid group},  Mat. Zametki 36 no.~1, 25–34 (1984) (in Russian); English translation: Math.\ notes Acad.\ Sci.\ USSR 36 no.~1-2, 505--510 (1984).

\bibitem{ArzhCashenGruberHume} {\bf Gulnara N. Arzhantseva, Christopher H. Cashen, Dominik Gruber, David Hume}, {\it Characterizations of Morse quasi-geodesics via superlinear divergence and sublinear contraction}, Doc.\ Math.\ 22, 1193--1224 (2017). 

\bibitem{ArzhCashenTao} {\bf Gulnara N. Arzhantseva, Chris H. Cashen, Jing Tao}, {\it Growth tight actions}, Pacific J.\ Math.~278 (2015), no. 1, 1--49.


\bibitem{Behrstock}{\bf Jason Behrstock}, Asymptotic geometry of the mapping class group and Teichmüller space, Geom.\ Topol.\ 10, 1523--1578 (2006)


\bibitem{Bestvina} {\bf Mladen Bestvina}, {\it Non-positively curved aspects  of Artin groups of finite type}, Geom.\ Topol.\ 3 (1999),  269--302.

\bibitem{BestvinaFujiwara} {\bf Mladen Bestvina, Koji Fujiwara}, {\it Bounded cohomology of subgroups of mapping class groups}, Geom.\ Topol.\ 6 (2002), 69--89

\bibitem{Birman-Gebhardt-GM} {\bf Joan Birman, Volker Gebhardt, Juan Gonz\'alez-Meneses}, {\it Conjugacy in Garside groups I: cyclings, powers and rigidity}, Groups Geom.\ Dyn.\ 1 (2007), 221--279


\bibitem{BradyMccammond} {\bf Tom Brady, Jon McCammond}, {\it Braids, posets and orthoschemes}, Algebr.\ Geom.\ Topol.~10, No.~4, 2277--2314 (2010)

\bibitem{BradyTran} {\bf Noel Brady, Hung Cong Tran}, {\it Divergence of finitely presented groups}, arXiv:2002.03653


\bibitem{BrieskornSaito} {\bf Egbert Brieskorn, Kyoji Saito}, {\it Artin-Gruppen und Coxeter-Gruppen}, Invent. Math. 17, (1972), 245--271.

\bibitem{CalvezEuclidean} {\bf Matthieu Calvez}, {\it Euclidean Artin-Tits groups are acylindrically hyperbolic}, arXiv:2010.13145.


\bibitem{CalvezWiestCurve} {\bf Matthieu Calvez, Bert Wiest}, {\it Curve graphs and Garside groups}, Geom.\ Dedicata 188(1) (2017), 195--213

\bibitem{CalvezWiestAH} {\bf Matthieu Calvez, Bert Wiest}, {\it Acylindrical hyperbolicity and Artin-Tits groups of spherical type}, Geom.\ Dedicata 191(1) (2017), 199--215

\bibitem{CalvezWiestSurvey} {\bf Matthieu Calvez, Bert Wiest}, {\it Hyperbolic structures for Artin-Tits groups of spherical type},
Contemp.\ Math.\ 766 (2021), 83--98

\bibitem{CEHLPT} {\bf James Cannon, David Epstein, Derek Holt, Silvio Levy, Michael Paterson, William Thurston}, {\it Word processing in groups}, Boston, MA etc.: Jones and Bartlett Publishers (1992).

\bibitem{CashenMorseStrongContr}{\bf Christopher H.~Cashen}, {\it Morse subsets of CAT(0) spaces are strongly contracting}, Geom.\ Dedicata 204 (2020), 311--314

\bibitem{CharneyArtinBiautom} {\bf Ruth Charney}, {\it Artin groups of finite type are biautomatic}, Math. Ann. 292 (1992), no. 4, 671--683. 

\bibitem{CharneyInjectivity} {\bf Ruth Charney}, {\it Injectivity of the positive monoid for some infinite type Artin groups}, Geometric Group Theory Down Under, edited by John Cossey, Charles F. Miller, Walter D. Neumann and Michael Shapiro, Berlin, New York: De Gruyter, 2011, pp. 103--118.

\bibitem{CharneyMeier} {\bf Ruth Charney, John Meier}, {\it The language of geodesics for Garside groups}, Math.\ Z.\ 248 (2004), 495--509.

\bibitem{CMW} {\bf Ruth Charney, John Meier, Kim Whittlesey}, {\it Bestvina's normal form complex and the homology of Garside groups},  Geom.\ Dedicata 105 (2004), 171--188.


\bibitem{DahmaniGuirardelOsin} {\bf François Dahmani, Vincent Guirardel, Denis Osin}, Hyperbolically embedded subgroups and rotating families in groups acting on hyperbolic spaces, Mem.\ Am.\ Math.\ Soc., 1156 (2016).  


\bibitem{Dehornoy} {\bf Patrick Dehornoy}, {\it Groupes de Garside}, Ann. Sc.Ec. Norm. Sup. 35 (2), (2002), 267--306. 

\bibitem{DDGKM} {\bf Patrick Dehornoy, Fran\c{c}ois Digne, Eddy Godelle, Daan Krammer, Jean Michel}, {\it Foundations of Garside Theory}, EMS Tracts in Mathematics, volume 22, European Mathematical Society, 2015.

\bibitem{DehornoyParis}{\bf Patrick Dehornoy, Luis Paris}, {\it Gaussian groups and Garside groups: two generalizations of Artin groups}, Proc. London Math. Soc. 79 (3), (1999), 569--604.

\bibitem{DuchinRafi}{\bf Moon Duchin, Kasra Rafi}, {\it Divergence of geodesics in Teichmüller space and the mapping class group},
Geom.\ Funct.\ Anal.~19, No.~3, (2009), 722--742. 

\bibitem{Deligne} {\bf Pierre Deligne}, {\it Les immeubles des groupes de tresses g\'en\'eralis\'es}, Invent.\ Math.\ 17 (1972), 273--302. 

\bibitem{ElRifaiMorton} {\bf Elsayed Elrifai, Hugh Morton}, {\it Algorithms for positive braids}, Quart.\ J.\ Math.\ Oxford (2), 45 (1994), 479--497.

\bibitem{Garside} {\bf Frank A. Garside}, {\it The braid group and other groups}, Quart.\ J.\ Math.\ 20 (1) (1969), 235--254.

\bibitem{GebhardtGM} {\bf Volker Gebhardt, Juan Gonz\'alez-Meneses}, {\it The cyclic sliding operation in Garside groups}, Math Z.\ 265 (2010), 85--114. 

\bibitem{GebhardtTawn} {\bf Volker Gebhardt, Steven Tawn}, {\it Zappa-Szép products of Garside monoids}, Math.\ Z.\ 282, No.\ 1--2, 341--369 (2016)

\bibitem{HaettelKielakSchwer} {\bf Thomas Haettel, Dawid Kielak, Petra Schwer}, {\it The six-strand braid group is CAT(0)}, Geom.\ Dedicata 182, 263--286 (2016)

\bibitem{HamenstadtRank1} {\bf Ursula Hamenst\"adt}, {\it{Rank-one isometries of proper CAT(0)-spaces}}, Contemp.\ Math.\ 501, 43--59 (2009). 

\bibitem{Jeong} {\bf Seong Gu Jeong}, {\it The seven-strand braid group is CAT(0)}, arXiv:2009.09350

\bibitem{MM1} {\bf Howard Masur, Yair Minsky}, {\it Geometry of the complex of curves I: Hyperbolicity}, Invent.\ Math.\ 138 (1999), 103--149.

\bibitem{MM2} {\bf Howard Masur, Yair Minsky}, {\it Geometry of the complex of curves II: Hierarchical Structure}, Geometric and Functional Analysis, 10(4)(2000), 902--974 .



\bibitem{MinskyQuasiProj} {\bf Yair Minsky}, {\it Quasi-projections in Teichm\"uller space}, J.\ Reine Angew.\ Math.\ 473, 121--136 (1996).


\bibitem{OsinAcylHyp} {\bf Denis Osin}, \textit{Acylindrically hyperbolic groups}, Trans.\ Amer.\ Math.\ Soc.\ 368 (2016), 851--888.

\bibitem{OsinActingAcyl} {\bf Denis Osin} {\it Groups acting acylindrically on hyperbolic spaces}, 
Sirakov, Boyan (ed.) et al., Proceedings of the international congress of mathematicians, ICM 2018, Rio de Janeiro, 
Volume II. Invited lectures
919--939 (2018). 

\bibitem{Picantin} {\bf Matthieu Picantin}, 
{\it The center of Garside groups}, J.\ Algebra 245-1 (2001) 92-–122. 

\bibitem{RafiVerberne} {\bf Kasra Rafi, Yvon Verberne}, {\it Geodesics in the mapping class group}, arXiv:1810.12489

\bibitem{SistoQuasiConv} {\bf Alessandro Sisto}, {\it Quasi-convexity of hyperbolically embedded subgroups}, Math.\ Z.\ 283, No. 3-4, 649--658 (2016). 
\bibitem{SistoRandomWalk} {\bf Alessandro Sisto}, {\it Contracting elements and random walks}, J.\ Reine Angew.\ Math.\ 742, 79--114 (2018). 

\bibitem{Sultan} {\bf Harold Sultan}, {\it Hyperbolic quasi-geodesics in CAT(0) spaces}, Geom.\ Dedicata 169 (2014), no. 1, 209--224.
\end{thebibliography}
\end{document}